\documentclass[11pt]{amsart}

\usepackage{amsmath, amssymb, amsthm, amsfonts,mathrsfs}

\input xy
\xyoption{all}

\usepackage{hyperref}

\swapnumbers
\numberwithin{equation}{section}

\theoremstyle{plain}
\newtheorem{theorem}[subsubsection]{Theorem}
\newtheorem{lemma}[subsubsection]{Lemma}
\newtheorem{prop}[subsubsection]{Proposition}
\newtheorem{cor}[subsubsection]{Corollary}
\newtheorem{conj}[subsubsection]{Conjecture}

\theoremstyle{definition}
\newtheorem{defn}[subsubsection]{Definition}

\newtheorem{remark}[subsubsection]{Remark}
\newtheorem{exam}[subsubsection]{Example}
\newtheorem{ques}[subsubsection]{Question}

\def\AA{\mathbb{A}}

\def\DD{\mathbb{D}}

\def\FF{\mathbb{F}}
\def\GG{\mathbb{G}}

\def\NN{\mathbb{N}}

\def\PP{\mathbb{P}}
\def\QQ{\mathbb{Q}}
\def\RR{\mathbb{R}}

\def\ZZ{\mathbb{Z}}

\newcommand\cB{\mathcal{B}}

\newcommand\cD{\mathcal{D}}
\newcommand\cE{\mathcal{E}}
\newcommand\cF{\mathcal{F}}
\newcommand\cG{\mathcal{G}}
\newcommand\cH{\mathcal{H}}

\newcommand\cL{\mathcal{L}}

\newcommand\cO{\mathcal{O}}
\newcommand\cP{\mathcal{P}}
\newcommand\cQ{\mathcal{Q}}

\newcommand\cV{\mathcal{V}}

\newcommand\cX{\mathcal{X}}
\newcommand\cY{\mathcal{Y}}
\newcommand\cZ{\mathcal{Z}}

\def\bH{\mathbf{H}}
\def\bI{\mathbf{I}}

\def\bK{\mathbf{K}}
\def\bL{\mathbf{L}}

\def\bP{\mathbf{P}}

\newcommand\frX{\mathfrak{X}}

\newcommand\fm{\mathfrak{m}}

\newcommand\frp{\mathfrak{p}}
\newcommand\frq{\mathfrak{q}}

\newcommand\tilW{\widetilde{W}}

\def\dG{\widehat{G}}

\newcommand\AJ{\textup{AJ}}

\newcommand{\Bun}{\textup{Bun}}

\newcommand{\Ch}{\textup{Ch}}

\newcommand{\coker}{\textup{coker}}

\newcommand\ev{\textup{ev}}

\newcommand{\Fl}{\textup{Fl}}
\newcommand{\fl}{f\ell}
\newcommand\Forg{\textup{Forg}}
\newcommand\Fr{\textup{Fr}}

\newcommand\Gal{\textup{Gal}}

\newcommand{\Gr}{\textup{Gr}}

\newcommand{\Herm}{\textup{Herm}}
\newcommand{\Hk}{\textup{Hk}}

\newcommand\IC{\textup{IC}}
\newcommand\id{\textup{id}}

\newcommand\inv{\textup{inv}}

\newcommand\Lie{\textup{Lie}\ }

\newcommand\Mot{\textup{Mot}}

\newcommand{\Nm}{\textup{Nm}}

\newcommand\Perv{\textup{Perv}}
\newcommand{\Pic}{\textup{Pic}}
\newcommand\pr{\textup{pr}}

\newcommand\Proj{\textup{Proj}}

\newcommand\Rep{\textup{Rep}}
\newcommand{\Res}{\textup{Res}}

\newcommand\rk{\textup{rk}}

\newcommand\Sht{\textup{Sht}}

\newcommand\Spec{\textup{Spec}\ }

\newcommand\triv{\textup{triv}}

\newcommand{\ur}{\textup{ur}}
\newcommand{\val}{\textup{val}}
\newcommand{\Vect}{\textup{Vect}}

\newcommand\Aut{\textup{Aut}}
\newcommand\Hom{\textup{Hom}}
\newcommand\End{\textup{End}}

\newcommand{\Ext}{\textup{Ext}}

\newcommand\GL{\textup{GL}}

\newcommand\PGL{\textup{PGL}}
\newcommand\Ug{\textup{U}}

\newcommand\SL{\textup{SL}}

\newcommand\Sp{\textup{Sp}}

\newcommand{\Gm}{\GG_m}

\newcommand{\Ad}{\textup{Ad}}

\newcommand\xch{\mathbb{X}^*}
\newcommand\xcoch{\mathbb{X}_*}

\newcommand{\incl}{\hookrightarrow}
\newcommand{\isom}{\stackrel{\sim}{\to}}

\newcommand{\surj}{\twoheadrightarrow}

\newcommand{\Ql}{\QQ_{\ell}}
\newcommand{\Qlbar}{\overline{\QQ}_\ell}

\newcommand{\hotimes}{\widehat{\otimes}}

\renewcommand{\j}[1]{\langle{#1}\rangle}
\newcommand{\wt}[1]{\widetilde{#1}}
\newcommand{\wh}[1]{\widehat{#1}}
\newcommand\quash[1]{}
\newcommand\un{\underline}

\newcommand{\ov}{\overline}
\newcommand{\bs}{\backslash}
\newcommand{\tl}[1]{[\![#1]\!]}
\newcommand{\lr}[1]{(\!(#1)\!)}
\newcommand\sss{\subsubsection}
\newcommand\xr{\xrightarrow}
\newcommand\op{\oplus}
\newcommand\ot{\otimes}
\newcommand\one{\mathbf{1}}

\renewcommand\c\circ
\newcommand{\vn}{\varnothing}

\newcommand{\cohog}[2]{\textup{H}^{#1}({#2})}     
\newcommand{\cohoc}[2]{\textup{H}_{c}^{#1}({#2})}     

\newcommand\upH{\textup{H}}

\newcommand{\oll}[1]{\overleftarrow{#1}}
\newcommand{\orr}[1]{\overrightarrow{#1}}

\renewcommand\a\alpha
\renewcommand\b\beta
\newcommand\g\gamma
\newcommand\G\Gamma
\renewcommand\d\delta
\newcommand\D\Delta
\newcommand{\e}{\epsilon}
\newcommand{\io}{\iota}
\newcommand{\ka}{\kappa}
\renewcommand{\th}{\theta}

\newcommand{\ph}{\varphi}
\renewcommand\r\rho
\newcommand{\s}{\sigma}
\newcommand{\Sig}{\Sigma}
\renewcommand{\t}{\tau}

\newcommand{\y}{\eta}
\newcommand{\z}{\zeta}

\newcommand{\vp}{\varpi}
\renewcommand{\l}{\lambda}
\renewcommand{\L}{\Lambda}
\newcommand{\om}{\omega}
\newcommand{\Om}{\Omega}

\newcommand\sh{\sharp}
\newcommand\da{\dagger}

\newcommand{\kbar}{\overline{k}}

\newcommand\Sch{\mathbf{Sch}}
\newcommand\Gpd{\mathbf{Gpd}}

\newcommand{\dr}{\dashrightarrow}
\newcommand{\dxr}{\dashrightarrow}
\newcommand{\br}{\breve}
\newcommand{\pl}{\partial}
\newcommand{\QGr}{\mathrm{QGr}}
\newcommand{\disj}{\mathrm{disj}}
\newcommand{\pos}{\mathrm{pos}}
\newcommand{\Sat}{\mathrm{Sat}}
\newcommand{\HN}{\mathrm{HN}}
\newcommand{\ab}{\mathrm{ab}}
\newcommand{\colim}{\mathrm{colim}}

\newcommand{\IsoSht}{\mathrm{IsoSht}}
\newcommand{\IsoCrys}{\mathrm{IsoCrys}}
\newcommand{\sep}{\mathrm{sep}}
\renewcommand{\div}{\mathrm{div}}
\newcommand{\ord}{\mathrm{ord}}
\newcommand{\Dr}{\mathrm{Dr}}
\newcommand{\Rees}{\mathrm{Rees}}

\newcommand\sL{\mathscr{L}}
\newcommand\HD{\textup{HD}}
\newcommand\GGP{\textup{GGP}}
\newcommand\KR{\textup{KR}}

\title{Introduction to Shtukas and their moduli}

\dedicatory{}
\author{Zhiwei Yun}
\thanks{Supported partially by the Simon Foundation and the Packard Foundation.}
\address{Department of Mathematics, Massachusetts Institute of Technology, 77 Massachusetts Ave, Cambridge, MA 02139}
\email{zyun@mit.edu}
\date{}

\begin{document}

\begin{abstract}
These are lectures notes of my talks at the IHES summer school on the Langlands program in 2022. We give an introduction to the notion of Shtukas,  their relation with more familiar geometric objects,  their moduli spaces and applications to automorphic forms. 
\end{abstract}

\maketitle
\tableofcontents

\section{Introduction}


The notion of a Shtuka was introduced by Drinfeld \cite{Dr-FBun} under the name ``$F$-bundles''. Here $F$ stands for Frobenius. Indeed a Shtuka is a bundle with a kind of Frobenius structure.  The word Shtuka has its origin in Russian, which means ``gadget'' or ``thing''. Shtukas generalize the earlier notion of ``elliptic modules'' introduced by Drinfeld, which are analogs of elliptic curves. General Shtukas can be thought of as a version of motives not defined in characteristic $p$ but whose coefficient field is of characteristic $p$.  

Moduli spaces or stacks of Shtukas are function field analogs of Shimura varieties. As such, they have been playing a key role in establishing cases of the Langlands correspondence over function fields. 

Drinfeld \cite{Dr-Pet, Dr-comp} used the moduli space of rank $2$ Shtukas to prove the Ramanujan-Petersson  conjecture and the Langlands conjecture for $\GL_{2}$ over a global function field.  

Laumon, Rapoport and Shtuhler \cite{LRS} used Shtukas for division algebras to prove part of the global Langlands conjecture for division algebras over a function field, and to prove the local Langlands conjecture for $\GL_{n}$ (any $n$) over local function fields.

L. Lafforgue \cite{LLaff} used the moduli stack of rank $n$ Shtukas and their compactifications to prove the Ramanujan-Petersson conjecture and the Langlands conjecture for $\GL_{n}$ over a global function field for any $n\in\NN$. 

Varshavsky \cite{Var} generalized the notion of Shtukas from the context of vector bundles to the context of principal $G$-bundles, for any split reductive group $G$ over the base field $\FF_{q}$. We shall call them $G$-Shtukas. He proved fundamental results about the geometry of the moduli stack of $G$-Shtukas.

More recently, V. Lafforgue \cite{VLaff} used Varshavsky's moduli of $G$-Shtukas, combined with several ingenious ideas (most importantly the Excursion Operators) to prove the ``automorphic to Galois'' direction of the global Langlands conjecture for an arbitrary connected reductive group $G$ over a function field. 

Although not within the scope of this article, there are local versions of Shtukas which play an important role in proving the local Langlands conjectures. In \cite{GL}, Genestier  and V. Lafforgue used local Shtukas to prove the ``representations to Galois'' direction of the local Langlands conjecture for an arbitrary connected reductive group $G$ over a local function field. Another exciting development is Scholze's theory of local Shtukas for mixed characteristic local fields, whose moduli spaces generalize Rapoport-Zink spaces, see \cite{S-Berkeley}.

In this survey article, we will introduce from the scratch the notion of Shtukas for a split reductive group over $\FF_{q}$ in \S\ref{s:def}. We will then describe its relationship with Drinfeld modules, and its analogy with Deligne-Lusztig varieties and motives over finite fields in \S\ref{s:sib}. Then in \S\ref{s:geom} and \S\ref{s:coho} we will study important geometric constructions on the moduli stack of Shtukas such as the Hecke correspondence and the partial Frobenius, and their implication on the cohomology of the moduli stacks. These serve as background for the lectures of C.Xue \cite{Xue}. In the last section  \S\ref{s:adv}, we mention several interesting topics that deserve further study, such as compactifications of the moduli stacks, and recent works relating special cycles on the moduli stacks of Shtukas to higher derivatives of $L$-functions and Eisenstein series.

There are excellent expository articles on Shtukas, for example \cite{Laumon} and \cite{NDT-Sht}, from which the author learned the theory. The research articles \cite{Var} and \cite{ARH} can serve as comprehensive references for details of proofs. Finally, large part of this article follows the treatment in the breakthrough paper \cite{VLaff}, from which the readers can find more details and applications to the Langlands program.  


\subsection*{Acknowledgement} 
I thank the organizers of the 2022 IHES Summer School on the Langlands Program for inviting me to speak on this topic. I thank Jianqiao Xia and the referee for carefully reading the draft and useful suggestions.

\section{Basic definitions}\label{s:def}

\subsection{$G$-torsors on curves}
Let $k$ be the finite field $\FF_{q}$. For two $k$-stacks $\cX$ and $\cY$, we use $\cX\times\cY$ to denote their fiber product over $\Spec k$.

Let $X$ be a geometrically connected smooth projective curve over $k$. Let $F=k(X)$ be the function field of $X$. Let $|X|$ be the set of closed points of $X$. For $x\in |X|$, let $\cO_{x}, \fm_{x}, F_{x}$ and $k_{x}$ denote the completed local ring of $X$ at $x$, its maximal ideal, fraction field and residue field respectively.

\sss{$G$-torsors} Let $G$ be a split connected reductive group over $k$. A $G$-torsor over a $k$-scheme $Y$ is a scheme $\cE$ with a right action of $G$ and a $G$-invariant map $\pi: \cE
\to Y$ such that \'etale locally over $Y$, $\cE$ is $G$-equivariantly isomorphic to $G\times Y$ with the right translation action of $G$.

Let $\Bun_{G}: \Sch_{k}\to \Gpd$ (where $\Sch_{k}$ is the category of $k$-schemes, and $\Gpd$ the category of groupoids) be the functor that assigns a $k$-scheme $S$ the groupoid of $G$-torsors over $X\times S:=X\times_{k}S$, where morphisms are  isomorphisms of $G$-torsors. Then $\Bun_{G}$ is representable by a smooth Artin stack locally of finite type over $k$. 

\begin{exam} Let $G=\GL_{n}$. There is an equivalence of categories between the groupoid of $\GL_{n}$-torsors over a scheme $Y$ and the groupoid of vector bundles of rank $n$ over $Y$. Therefore, $\Bun_{G}$ is isomorphic to the moduli stack $\Bun_{n}$ of rank $n$ vector bundles over $X$.

If $G=\SL_{n}$, then $\Bun_{G}(S)$ parametrizes pairs $(\cV, \l)$ where $\cV$ is a vector bundle of rank $n$ over $X\times S$, and $\l$ is an isomorphism $\l: \cO_{X\times S}\isom \det(\cV):=\wedge^{n}(\cV)$.
\end{exam}

\begin{exam} Let $G=\Sp_{2n}$. There is an equivalence of categories between $\Sp_{2n}$-torsors over a scheme $Y$ and the groupoid of  symplectic vector bundles $(\cV, \om)$ over $Y$. Here $\cV$ is a vector bundle of rank $2n$ over $Y$, and $\om: \cV\ot \cV\to \cO_{Y}$ is a perfect alternating pairing. Therefore, $\Bun_{G}$ is isomorphic to the moduli stack  of rank $2n$ symplectic vector bundles over $X$.
\end{exam}

\sss{Level structure}\label{sss:level}

Let $D\subset X$ be a finite subscheme, and let $S$ be a $k$-scheme. A $D$-level structure of a $G$-torsor $\cE$ over $X\times S$ is a trivialization of the $G$-torsor $\cE|_{D\times S}$ over $D\times S$ obtained by restriction.  Let $\Bun_{G,D}: \Sch_{k}\to \Gpd$ be the functor whose value at $S$ is the groupoid of pairs $(\cE,\io_{D})$, where $\cE$ is a $G$-torsor over $X\times S$ and $\io_{D}$ is a $D$-level structure of $\cE$. Then $\Bun_{G,D}$ is representable by an Artin stack, and the forgetful map $\Forg_{D}: \Bun_{G,D}\to \Bun_{G}$ is a $G_{D}:=\Res_{D/k}(G\times D)$-torsor.

\begin{exam} Let $G=\Sp_{2n}$, and let $D=x$ for some $x\in |X|$. For any $k$-scheme $S$,  $\Bun_{G,x}(S)$ is the groupoid of tuples $(\cV, \om, \{e_{1},\cdots, e_{n}, f_{1},\cdots, f_{n}\})$, where $(\cV,\om)$ is a symplectic vector bundle of rank $2n$ over $X\times S$, $e_{i}$ and $f_{i}$ are sections of $\cV|_{D\times S}$, such that $\om(e_{i},e_{j})=0$, $\om(f_{i},f_{j})=0$ and $\om(e_{i},f_{j})=\d_{ij}$ for all $1\le i,j\le n$ (all these are equalities of sections of $\cO_{D\times S}$).  
\end{exam}

It is also useful to consider other level structures. Let $x\in |X|$. Let $L^{+}_{x}G$ be the formal arc group of $G$ at $x$. This is the group scheme over $k_{x}$ whose $R$-points ($R$ is a $k_{x}$-algebra) are $G(R\hotimes_{k_{x}} \cO_{x})$. Let $L_{x}G$ be the formal loop group of $G$ at $x$. This is the group ind-scheme over $k_{x}$ whose $R$-points are $G(R\hotimes_{k_{x}} F_{x})$. See \cite[\S1.a]{PR}. 

For $n\in \NN$, let $D_{nx}=\Spec \cO_{x}/\fm^{n}_{x}$. We have the truncated version $G_{nx}:=\Res_{D_{nx}/k_{x}}(G\times D_{nx})$ of $L^{+}_{x}G$.  
Let $r_{nx}: L^{+}_{x}G\to G_{nx}$ be the reduction map.  Let
\begin{equation*}
L^{(n)}_{x}G=\ker(r_{nx})
\end{equation*}
be the $n$th principal congruence subgroup of $L^{+}_{x}G$. 

Let $\bK_{x}\subset L^{+}_{x}G$ be a subgroup scheme over $k_{x}$ of the form $r_{nx}^{-1}(\ov\bK_{x,n})$ for some $n\in\NN$ and $k_{x}$-subgroup $\ov\bK_{x,n}\subset G_{nx}$. Then we define $\Bun_{G,\bK_{x}}$ to be the quotient stack
\begin{equation*}
\Bun_{G,\bK_{x}}:=\Bun_{G, nx}/(\Res_{k_{x}/k}\ov\bK_{x,n}).
\end{equation*}
Note that $\Res_{k_{x}/k}G_{nx}=\Res_{D_{nx}/k}(G\times D_{nx})$ acts on $\Bun_{G, nx}$, hence so does its subgroup $\Res_{k_{x}/k}\ov\bK_{x,n}$. It is easy to see that $\Bun_{G,\bK_{x}}$ is independent of the choice of $n$ up to canonical isomorphisms.


\begin{exam} Let $\bI_{x}$ be an Iwahori subgroup of $L^{+}_{x}G$, i.e., it is the preimage of a Borel subgroup $B_{x}\subset G\ot k_{x}$ under $r_{x}: L^{+}_{x}G\to G\ot k_{x}$. Then an $S$-point of $\Bun_{G,\bI_{x}}$ is the same datum as a pair $(\cE, \cF_{x})$ where $\cE$ is a $G$-torsor over $X\times S$ and $\cF_{x}$ is a Borel reduction of the restriction $\cE|_{\{x\}\times S}$, i.e., a section of the bundle $\cE|_{\{x\}\times S}\times^{G}\fl_{G}\to \{x\}\times S$ ($\fl_{G}$ is the flag variety of $G$).  From the latter description we see that  $\Bun_{G,\bI_{x}}$ is canonically independent of the choice the Borel subgroup $B_{x}$. 
\end{exam}

More generally, one can define the moduli stack  $\Bun_{G,\bK_{x}}$ when $\bK_{x}\subset L_{x}G$ is not necessarily contained in $L^{+}_{x}G$, but contains a congruence subgroup of $L^{+}_{x}G$ as a normal subgroup with finite codimension. We call such $\bK_{x}$ a {\em level group} at $x$. To construct $\Bun_{G,\bK_{x}}$,  we let $\Bun_{G, \infty x}$ be the inverse limit of the tower $(\Bun_{G,nx})_{n\in\NN}$.  It can be shown that the $L^{+}_{x}G$-action on $\Bun_{G,\infty x}$ extends canonically to an action of $L_{x}G$. We then define $\Bun_{G,\bK_{x}}$ to be the quotient $\Bun_{G,\infty x}/(\Res_{k_{x}/k}\bK_{x})$. To see this is an Artin stack, we choose a principal congruence subgroup $L^{(n)}_{x}G$ normal in $\bK_{x}$, and let $\ov \bK_{x,n}=\bK_{x}/L^{(n)}_{x}G$, and set $\Bun_{G,\infty x}/(\Res_{k_{x}/k}\bK_{x})$ to be $\Bun_{G,nx}/(\Res_{k_{x}/k}\ov\bK_{x,n})$. In particular, we can apply the above construction to parahoric subgroups $\bP_{x}\subset L_{x}G$. 

\begin{remark} For a general level structure $\bK_{x}$,  it is customary to call an $S$-point of $\Bun_{G,\bK_{x}}$ a $G$-torsor over $X\times S$ with a $\bK_{x}$-level structure at $x$. However, this terminology can be misleading: when $\bK_{x}$ is not necessarily contained in $L^{+}_{x}G$, an $S$-point of $\Bun_{G,\bK_{x}}$ does not induce a $G$-torsor on the whole of $X\times S$, only on $(X-\{x\})\times S$. The following Example illustrates this point.
\end{remark}

\begin{exam}\label{ex:SL parahoric} Let $G=\SL_{n}$ and fix $x\in |X|$. Then up to conjugacy $L_{x}G$ has $n$ maximal parahoric subgroups $\bP^{(i)}_{x}$, parametrized by $i\in \ZZ/n\ZZ$. If we let $\wt G=\GL_{n}$, then we can take $\bP^{(i)}_{x}=g_{i}(L^{+}_{x}\wt G)g_{i}^{-1}\cap L_{x}G$, where $g_{i}$ is any element in $\wt G(F_{x})$ satisfying $\val_{F_{x}}(\det(g_{i}))\equiv i$ mod $n$. 

We have the following interpretation of $\Bun_{G, \bP^{(i)}_{x}}$. Consider the map $\det: \Bun_{\wt G}\to \Pic_{X}$ sending $\cV$ to $\wedge^{n}(\cV)$. Then there is a canonical isomorphism of stacks
\begin{equation*}
\Bun_{G, \bP^{(i)}_{x}}\cong {\det}^{-1}(\cO(-ix)), \quad i\in \ZZ.
\end{equation*}
Of course tensoring with $\cO(x)$ identifies $\det^{-1}(\cO(-ix))$ and $\det^{-1}(\cO(-(i-n)x))$, so the isomorphism class of the right side above depends only on the image of $i$ in $\ZZ/n\ZZ$. 

In other words, for $i\in \ZZ$, the $S$-points of $\Bun_{G, \bP^{(i)}_{x}}$ classify pairs $(\cV,\l)$ where $\cV$ is a rank $n$ vector bundle over $X\times S$, and $\l$ is an isomorphism $\cO(-ix)\isom \wedge^{n}(\cV)$.
\end{exam}

The above definitions clearly extend to the case of level structures at not just one but finitely many points. If $\Sig\subset |X|$ is a finite subset, and $\bK:=(\bK_{x})_{x\in \Sigma}$ is a collection of level groups at each $x\in\Sig $, we denote the resulting moduli stack by $\Bun_{G,\bK}$. 


\subsection{Hecke stack for $G$-torsors}
In the following we fix a finite set of places $\Sigma\subset |X|$ and a level group $\bK_{x}$ for each $x\in \Sigma$.  Let $U=X-\Sigma$. We denote by $\bK$ the collection of level groups $(\bK_{x})_{x\in \Sigma}$. The moduli stack $\Bun_{G,\bK}$ is defined.

\sss{Hecke stack}\label{sss:Hk}

Let $I$ be a finite set. Let $\Hk^{I}_{G,\bK}: \Sch_{k}\to \Gpd$ be the functor whose value at a $k$-scheme $S$ is the groupoid of tuples $\xi=((x_{i})_{i\in I}, \cE_{0},\cE_{1} ,\a)$ where
\begin{enumerate}
\item For each $i\in I$, $x_{i}: S\to U$ is a morphism with graph $\G(x_{i})\subset X\times S$. The $x_{i}$'s are called the {\em legs} of $\xi$.
\item $\cE_{0},\cE_{1}\in \Bun_{G,\bK}(S)$, i.e.,  $\cE_{0}$ and $\cE_{1}$ are  $G$-torsors over $X\times S$ with $\bK_{x}$-level structures along $\{x\}\times S$.
\item $\a$ is an isomorphism of $G$-torsors with $\bK$-level structures $\cE_{0}|_{X\times S-\cup_{i\in I}\G(x_{i})}\isom \cE_{1}|_{X\times S-\cup_{i\in I}\G(x_{i})}$.
\end{enumerate}

By definition, $\Hk_{G,\bK}^{I}$ is equipped with a map 
\begin{equation*}
\Hk^{I}_{G,\bK}\to U^{I}.
\end{equation*}

When $I=\vn$, $\Hk^{I}_{G,\bK}=\Bun_{G,\bK}$. When $\Sig=\vn$ (so $\bK$ is no data), we simply write $\Hk^{I}_{G,\bK}$ as $\Hk^{I}_{G}$.


\sss{Iterated version of Hecke stacks}\label{sss:it Hk}
Let $I=I_{1}\coprod \cdots\coprod I_{r}$ be a decomposition of a finite set $I$. We have a variant $\Hk^{(I_{1},\cdots, I_{r})}_{G,\bK}$ of the Hecke stack $\Hk^{I}_{G,\bK}$ by refining the isomorphism $\a$ into $r$ steps. An $S$-point of $\Hk^{(I_{1},\cdots, I_{r})}_{G,\bK}$ consists of tuples $((x_{i})_{i\in I}, (\cE_{j})_{0\le j\le r} ,(\a_{j})_{1\le j\le r})$ where
\begin{enumerate}
\item For each $i\in I$, $x_{i}: S\to U$ is a morphism with graph $\G(x_{i})\subset X\times S$.
\item $\cE_{0},\cdots, \cE_{r}\in \Bun_{G,\bK}(S)$.
\item For $1\le j\le r$, $\a_{j}$ is an isomorphism of $G$-torsors with $\bK$-level structures 
$$\a_{j}: \cE_{j-1}|_{X\times S-\cup_{i\in I_{j}}\G(x_{i})}\isom \cE_{j}|_{X\times S-\cup_{i\in I_{j}}\G(x_{i})}.$$
\end{enumerate}

Recording the bundle $\cE_{j}$ gives a morphism
\begin{equation*}
p_{j}: \Hk^{(I_{1},\cdots, I_{r})}_{G,\bK}\to \Bun_{G,\bK}, \quad j=0,1,\cdots, r.
\end{equation*}

Forgetting the intermediate bundles $\cE_{1},\cdots, \cE_{r-1}$ and taking the composition $\a=\a_{r}\c\cdots \c\a_{1}$, we get a map
\begin{equation*}
\pi^{(I_{1},\cdots, I_{r})}_{I}: \Hk^{(I_{1},\cdots, I_{r})}_{G,\bK}\to \Hk^{I}_{G,\bK}.
\end{equation*}
This map is an isomorphism over the open subscheme $U^{(I_{1},\cdots, I_{r})}_{\disj}\subset U^{I}$ consisting of $(x_{i})_{i\in I}$ such that for $i$ and $i'$ belonging to different parts $I_{j}$ and $I_{j'}$, the graphs of $x_{i}$ and $x_{i'}$ are disjoint.

More generally, if each $I_{j}$ has a decomposition $I_{j}=I_{j,1}\coprod \cdots\coprod I_{j,s_{j}}$, then $I$ has a finer decomposition $I=I_{1,1}\coprod \cdots\coprod I_{1,s_{1}}\coprod I_{2,1}\coprod\cdots\coprod I_{2,s_{2}}\coprod \cdots\coprod  I_{r,s_{r}}$. We denote this refinement by $I=I'_{1}\coprod \cdots\coprod I'_{s}$ (so $s=s_{1}+\cdots+s_{r}$), where the parts are listed in the same order as above. We have a map
\begin{equation*}
\pi^{(I'_{1},\cdots, I'_{s})}_{(I_{1},\cdots, I_{r})}: \Hk^{(I'_{1},\cdots, I'_{s})}_{G,\bK}\to \Hk^{(I_{1},\cdots, I_{r})}_{G,\bK}
\end{equation*}
by recording only the bundles $\cE_{0},\cE_{s_{1}},\cE_{s_{1}+s_{2}},\cdots, \cE_{s}$ and the compositions $\a_{j}:=\a_{j,s_{j}}\c\cdots\c\a_{j,1}: \cE_{j-1}\dr\cE_{j}$, for $1\le j\le r$. This map is an isomorphism over $\prod_{j=1}^{r}U^{(I_{j,1},\cdots, I_{j,s_{j}})}_{\disj}\subset \prod_{j=1}^{r}U^{I_{j}}=U^{I}$.

\sss{Affine Grassmannian}\label{sss:Gr} We refer the readers to X. Zhu's expository article \cite{Zhu-Satake} for more details on this topic. The affine Grassmannian $\Gr_{G}$ of $G$ is by definition the ind-scheme over $k$ that represents the fpqc sheafification of the functor $R\mapsto (LG)(R)/(L^{+}G)(R)=G(R\lr{t})/G(R\tl{t})$ on $k$-algebras. Alternatively, if we let $D_{R}=\Spec R\tl{t}$ and $D^{\times}_{R}=\Spec R\lr{t}$, then $\Gr_{G}(R)$ classifies pairs $(\cE, \io)$ where $\cE$ is a $G$-torsor over $D_{R}$ and $\io$ is a trivialization of $\cE|_{D^{\times}_{R}}$, i.e., an isomorphism $\io: \cE|_{D^{\times}_{R}}\isom \cE_{\triv}|_{D^{\times}_{R}}$, where $\cE_{\triv}$ is the trivial $G$-torsor over $D_{R}$.

Let $T\subset G$ be a split maximal torus with Weyl group $W$.  We recall that the $L^{+}G$-orbits on the affine Grassmannian  $\Gr_{G}$ are indexed by Weyl group orbits of coweights $\xcoch(T)/W$: for $\l\in \xcoch(T)$ we assign the $L^{+}G$-orbit containing $t^{\l}\in LT$, which depends only on the $W$-orbit of $\l$. 

We choose a Borel subgroup containing $T$ which defines dominant coweights $\xcoch(T)^{+}$. Then we have a canonical bijection $\xcoch(T)^{+}\isom \xcoch(T)/W$, which allows us to index any $L^{+}G$-orbits on $\Gr_{G}$ uniquely by an element in $\xcoch(T)^{+}$. For $\l\in \xcoch(T)^{+}$, let $\Gr_{G, \l}$ be the $L^{+}G$-orbit of $t^{\l}$, as a locally closed reduced subscheme. We have $\dim\Gr_{G,\l}=\j{2\r,\l}$, where $2\r$ is the sum of positive roots. 

Let $\le$ be the partial order on $\xcoch(T)^{+}$ defined as follows: $\l'\le \l$ if and only if $\l-\l'$ is a sum of positive coroots with respect to $B$.  It is known that the closure of $\Gr_{G,\l}$ is $\Gr_{G,\le \l}=\cup_{\l'\in\xcoch(T)^{+}, \l'\le \l}\Gr_{G,\l'}$.

For a closed point $x\in |X|$, we similarly define $\Gr_{G,x}$ to be the ind-scheme representing the fpqc sheafification of $R\mapsto \Gr_{G,x}(R)=(L_{x}G)(R)/(L^{+}_{x}G)(R)$. After choosing a uniformizer $\vp_{x}$ at $x$, we get an isomorphism $\Gr_{G,x}\cong \Gr_{G}\ot_{k}k_{x}$. As $x$ varies, $\Gr_{G,x}$ form a family over $X$, which is a special case of the Beilinson-Drinfeld Grassmannian that we recall now.

\begin{exam} When $G=\GL_{n}$, $\Gr_{G}(k)$ is the set of free $k\tl{t}$-submodules of rank $n$ inside $k\lr{t}^{\op n}$. For $\l=(d_{1},d_{2},\cdots, d_{n})$ such that $d_{1}\ge d_{2}\ge\cdots \ge d_{n}$ are integers, viewed as a dominant coweight of the diagonal torus $T\subset G$, the point $t^{\l}$ corresponds to the $k\tl{t}$-submodule 
\begin{equation*}
t^{d_{1}}k\tl{t}\op t^{d_{2}}k\tl{t}\op\cdots\op t^{d_{n}}k\tl{t}\subset k\lr{t}^{\op n}.
\end{equation*}

When $d_{n}\ge0$, $\Gr_{G,\l}(k)$ is the set of $k\tl{t}$-submodules $\L\subset k\tl{t}^{\op n}$ such that $k\tl{t}^{\op n}/\L\cong \op_{i=1}^{n}k\tl{t}/(t^{d_{i}})$ as $k\tl{t}$-modules. 

When $\l=(d,0,\cdots, 0)$ with $d\ge0$, $\Gr_{G,\le\l}(k)$ is the set of $k\tl{t}$-submodules $\L\subset k\tl{t}^{\op n}$ such that $\dim_{k}k\tl{t}^{\op n}/\L=d$ (not $\le d$).
\end{exam}

\sss{Beilinson-Drinfeld Grassmannian}\label{sss:BD}
Let $\cE_{\triv}\in \Bun_{G}(k)$ be the trivial $G$-torsor on $X$. Define the {\em Beilinson-Drinfeld Grassmannian} $\Gr^{(I_{1},\cdots, I_{r})}_{G}$ using the Cartesian diagram
\begin{equation*}
\xymatrix{\Gr^{(I_{1},\cdots, I_{r})}_{G}\ar[r]\ar[d] & \Hk^{(I_{1},\cdots, I_{r})}_{G}\ar[d]^{p_{r}} \\
\Spec k\ar[r]^-{\cE_{\triv}} & \Bun_{G} 
}
\end{equation*}
Then we have the leg map $\Gr^{(I_{1},\cdots, I_{r})}_{G}\to X^{I}$. When $I$ is the singleton set $\{1\}$,  the fiber of $\Gr^{\{1\}}_{G}\to X$ at a closed point $x\in |X|$ is canonically isomorphic to $\Gr_{G,x}$. 

\sss{The relative position map}
We can form the quotient stack $\QGr_{G}:=L^{+}G\bs \Gr_{G}$, which is a local version of $\Hk^{\{1\}}_{G}$:  for a $k$-algebra $R$, an $R$-point of $\QGr_{G}$ is a triple $(\cE_{0}, \cE_{1}, \a)$ where $\cE_{i}$ are $G$-torsors over $D_{R}=\Spec R\tl{t}$, and $\a: \cE_{0}|_{D^{\times}_{R}}\isom \cE_{1}|_{D^{\times}_{R}}$. The points of $\QGr_{G}$ are thus in bijection with $\xcoch(T)^{+}$. For $\l\in \xcoch(T)^{+}$, let $\QGr_{G, \le \l}$ be the image of $\Gr_{G,\le \l}$, which is a closed reduced substack of $\QGr_{G}$. 

Let $\Aut(D)$ be the pro-algebraic group over $k$ of automorphisms of $k\tl{t}$. Then there is a right action of $\Aut(D)$ on $\QGr_{G}$ by change of coordinates: $\th: D\to D$ sends $(\cE_{0},\cE_{1}, \a)$ to $(\th^{*}\cE_{0}, \th^{*}\cE_{1}, \th^{*}\a)$.  Clearly $\QGr_{G, \le\l}$ is invariant under the action of $\Aut(D)$.

For the Hecke stack $\Hk^{\{1\}}_{G,\bK}$, we define an evaluation map 
\begin{equation*}
\ev^{\{1\}}_{\Hk}: \Hk^{\{1\}}_{G,\bK}\to \QGr_{G}/\Aut(D)
\end{equation*}
as follows. It suffices to construct an $\Aut(D)$-torsor  $\wt\Hk^{\{1\}}_{G,\bK}\to \Hk^{\{1\}}_{G,\bK}$ together with an $\Aut(D)$-equivariant map $\wt\ev^{\{1\}}_{\Hk}: \wt\Hk^{\{1\}}_{G,\bK}\to \QGr_{G}$. We define $\wt\Hk^{\{1\}}_{G,\bK}$ to be the moduli stack whose $R$-points are $(x, t_{x}, \cE_{0}, \cE_{1}, \a)$ where $(x,\cE_{0},\cE_{1}, \a)\in\Hk^{\{1\}}_{G,\bK}(R)$,  and $t_{x}$ is a regular function on the formal completion of $X\ot R$ along $\G(x)$ that is a uniformizer of $\wh\cO_{X\times s, x_{s}}$ for every geometric point of $s\in \Spec R$. Then $t_{x}$ gives a continuous isomorphism $\wh \cO_{X\ot R, \G(x)}\cong R\tl{t}$, hence a map $t^{\sh}_{x}: \Spec R\tl{t}\to X\ot R$. The map $\wt\ev^{\{1\}}_{\Hk}$ then sends $(x, t_{x}, \cE_{0}, \cE_{1}, \a)$ to $(t^{\sh, *}_{x}\cE_{0}, t^{\sh, *}_{x}\cE_{1}, t^{\sh, *}_{x}\a)$.

\sss{Bounded Hecke stacks}\label{sss:bd Hk}
For $\l\in \xcoch(T)^{+}$, we define a closed substack
\begin{equation*}
\Hk^{\{1\}, \le \l}_{G,\bK}\subset \Hk^{\{1\}}_{G,\bK}
\end{equation*}
to be the preimage of $\QGr_{G,\le \l}/\Aut(D)\subset \QGr_{G}/\Aut(D)$ under the map $\ev^{\{1\}}_{\Hk}$ above. We informally say that an $S$-point $(x,\cE_{0},\cE_{1}, \a)$ of $\Hk^{\{1\}}_{G,\bK}$ lies in $\Hk^{\{1\}, \le \l}_{G,\bK}$ if and only if the relative position of $\a$ is $\le \l$.

Now let $I$ be a finite set with a decomposition $I=I_{1}\coprod  I_{2}\coprod \cdots\coprod  I_{r}$. Let $\un\l=(\l_{i})_{i\in I}\in (\xcoch(T)^{+})^{I}$. We would like to define a closed substack $\Hk^{(I_{1},\cdots, I_{r}), \le \un\l}_{G,\bK}$ of $\Hk^{(I_{1},\cdots, I_{r})}_{G,\bK}$. The construction is done in several steps:
\begin{enumerate}
\item Let $U^{I_{j}}_{\disj}\subset U^{I_{j}}$ be the open subset of $(x_{i})_{i\in I_{j}}$ such that $\G(x_{i})$ are disjoint. Then we have a map
\begin{equation}\label{ev disj}
\ev^{(I_{1},\cdots, I_{r})}_{\Hk,\disj}: \Hk^{(I_{1},\cdots, I_{r})}_{G,\bK}|_{U^{I_{1}}_{\disj}\times\cdots\times U^{I_{r}}_{\disj}}\to (\QGr_{G}/\Aut(D))^{I}
\end{equation}
whose $i$th factor records the relative position of $\a_{j}$ (where $i\in I_{j}$) along $\G(x_{i})$. We define 
\begin{equation*}
\Hk^{(I_{1},\cdots, I_{r}), \le \un\l}_{G,\bK, \disj}\subset \Hk^{(I_{1},\cdots, I_{r})}_{G,\bK}|_{U^{I_{1}}_{\disj}\times\cdots\times U^{I_{r}}_{\disj}}
\end{equation*}
to be the preimage of $\prod_{i\in I}\QGr_{G,\le \l_{i}}/\Aut(D)$ under $\ev^{(I_{1},\cdots, I_{r})}_{\Hk,\disj}$.

\item We then define $\Hk^{(I_{1},\cdots, I_{r}), \le \un\l}_{G,\bK}$ to be the closure of $\Hk^{(I_{1},\cdots, I_{r}), \le \un\l}_{G,\bK, \disj}$ in $\Hk^{(I_{1},\cdots, I_{r})}_{G,\bK}$. 
\end{enumerate}

We define the bounded version of the Beilinson-Drinfeld Grassmannian $\Gr^{(I_{1},\cdots, I_{r}), \le\un\l}_{G}$ to be the preimage of $\Hk^{(I_{1},\cdots, I_{r}), \le \un\l}_{G}$ under the canonical map $\Gr^{(I_{1},\cdots, I_{r})}_{G}\to \Hk^{(I_{1},\cdots, I_{r})}_{G}$.

\begin{exam} Let $G=\GL_{n}$, and $\l_{i}=(1,0,\cdots, 0)$ for all $i\in I$. 
\begin{enumerate}
\item Iterated version. Let $I=\{1,\cdots, r\}$ decomposed into singletons $I_{i}=\{i\}$. Then the $S$-points of $\Hk^{(\{1\},\cdots,\{ r\}),\le\un\l}_{G}$ are 
$$((x_{i})_{1\le i\le r}, (\cV_{j})_{0\le j\le r}, (\a_{j})_{1\le j\le r}),$$
where $\cV_{j}$ are rank $n$ vector bundles over $X\times S$, and $\a_{j}: \cV_{j-1}\incl \cV_{j}$ is an injection of coherent sheaves whose cokernel is locally free of rank one on the graph $\G(x_{j})$, for $1\le j\le r$.

\item Non-iterated version. The $S$-points of $\Hk^{I,\le\un\l}_{G}$ are $((x_{i})_{i\in I}, \cV_{0}, \cV_{1}, \a)$ where $\cV_{0}$ and $\cV_{1}$ are rank $n$ vector bundles over $X\times S$, and $\a: \cV_{0}\incl \cV_{1}$ is an injection of coherent sheaves such that the divisor of $\det(\a): \det(\cV_{0})\to \det(\cV_{1})$ is  $\sum_{i\in I}\G(x_{i})$ as a divisor of $X\times S$.
\end{enumerate}

\end{exam}

\begin{remark} Our notation differs from that of \cite[\S1]{VLaff}. The Hecke stack $\Hk^{(I_{1},\cdots, I_{r})}_{G,\bK}$ was denoted $\textup{Hecke}_{\bK,I}^{(I_{1},\cdots, I_{r})}$ in {\em loc.cit}.  In \cite[Definition 1.2]{VLaff}, two bounded versions of the iterated Hecke stack were introduced, which were denoted $\textup{Hecke}^{(I_{1},\cdots, I_{r})}_{\bK, I, \un\l}$ and $\textup{Hecke}^{(I_{1},\cdots, I_{r})}_{\bK, I, \lesssim\un\l}$.  It can be shown that our Hecke stack $\Hk^{(I_{1},\cdots, I_{r}),\le\un\l}_{G, \bK}$ is the reduced structure of $\textup{Hecke}^{(I_{1},\cdots, I_{r})}_{\bK, I, \un\l}$, which is in turn contained in $\textup{Hecke}^{(I_{1},\cdots, I_{r})}_{\bK, I, \lesssim\un\l}$. 
\end{remark}

\subsection{$G$-Shtukas}\label{ss:Sht}
For any scheme $Y$ over $k=\FF_{q}$, denote by $\Fr_{Y}$ or simply $\Fr:Y\to Y$ the Frobenius morphism that raises regular functions to its $q$th power.

Again we fix a finite set of places $\Sigma\subset |X|$ and level groups $\bK=(\bK_{x})_{x\in \Sig}$.  Let $U=X-\Sigma$.

\sss{Definition of $\Sht_{G}$}
Let $I$ be a finite set with a decomposition $I=I_{1}\coprod \cdots\coprod I_{r}$. Define an Artin stack $\Sht_{G,\bK}^{(I_{1},\cdots, I_{r})}$ by the Cartesian diagram
\begin{equation}\label{defn Sht}
\xymatrix{\Sht_{G,\bK}^{(I_{1},\cdots, I_{r})}\ar[r]\ar[d] & \Hk_{G,\bK}^{(I_{1},\cdots, I_{r})}\ar[d]^{(p_{0},p_{r})}\\
\Bun_{G,\bK}\ar[r]^-{(\id,\Fr)} & \Bun_{G,\bK}\times \Bun_{G,\bK}
}
\end{equation}

From the construction, we have the map recording the legs
\begin{equation*}
\pi:=\pi^{(I_{1},\cdots, I_{r})}_{G,\bK}: \Sht_{G,\bK}^{(I_{1},\cdots, I_{r})}\to U^{I}.
\end{equation*}

\sss{Functor of points}\label{sss:Sht points}
By definition, for any test scheme $S$ over $k$, the $S$-points of $\Sht_{G,\bK}^{I}$ is the groupoid of tuples
\begin{equation*}
\xi=((x_{i})_{i\in I}, \cE, \a: \cE\dr {}^{\t}\cE)
\end{equation*}
where
\begin{itemize}
\item For $i\in I$, $x_{i}\in U(S)$ with graph $\G(x_{i})\subset U\times S$. The $x_{i}$'s are called the {\em legs} of the Shtuka $\xi$.
\item $\cE$ is a $G$-torsor over $X\times S$ with  $\bK_{x}$-level structure along $\{x\}\times S$.
\item ${}^{\t}{\cE}=(\id_{X}\times \Fr_{S})^{*}\cE$.
\item $\a$ is an isomorphism of $G$-torsors $\cE|_{X\times S-\cup_{i\in I}\G(x_{i})}\isom ({}^{\t}{\cE})|_{X\times S-\cup_{i\in I}\G(x_{i})}$ respecting the $\bK$-level structures.
\end{itemize}

More generally, for a decomposition $I=I_{1}\coprod \cdots\coprod I_{r}$, the $S$-points of $\Sht_{G,\bK}^{(I_{1},\cdots, I_{r})}$ is the groupoid of tuples
\begin{equation*}
((x_{i})_{i\in I},(\cE_{j})_{0\le j\le r}, (\a_{j}: \cE_{j-1}\dr \cE_{j})_{1\le j\le r}, \io)
\end{equation*}
where
\begin{itemize}
\item For $i\in I$, $x_{i}\in U(S)$ with graph $\G(x_{i})\subset X\times S$.
\item For $0\le j\le r$, $\cE_{j}$ is a $G$-torsor over $X\times S$ with  $\bK_{x}$-level structure along $\{x\}\times S$.
\item For $1\le j\le r$, $\a_{j}$ is an isomorphism of $G$-torsors respecting the $\bK$-level structures
\begin{equation*}
\cE_{j-1}|_{X\times S-\cup_{i\in I_{j}}\G(x_{i})}\isom \cE_{j}|_{X\times S-\cup_{i\in I_{j}}\G(x_{i})}.
\end{equation*}
\item $\io$ is an isomorphism of $G$-torsors respecting the $\bK$-level structures.
\begin{equation*}
\io:\cE_{r}\isom {}^{\t}\cE_{0}=(\id_{X}\times \Fr_{S})^{*}\cE_{0}.
\end{equation*}
\end{itemize}

\begin{exam}\label{ex:Sht no leg}
When $I=\vn$, by definition $\Sht^{\vn}_{G,\bK}$ is the fiber product of the diagonal $\D: \Bun_{G,\bK}\to \Bun_{G,\bK}^{2}$ with the graph of Frobenius $(\id,\Fr): \Bun_{G,\bK}\to \Bun_{G,\bK}^{2}$.  As a general fact, for any Artin stack $\frX$ over $k$, the fiber product of $\D: \frX\to \frX^{2}$ with the graph of Frobenius $(\id,\Fr):\frX\to \frX^{2}$ is isomorphic to the groupoid $\frX(k)$, viewed as a discrete stack over $k$. In our case we have 
\begin{equation*}
\Sht^{\vn}_{G,\bK}=\coprod_{\cE\in \Bun_{G,\bK}(k)}[(\Spec k)/\Aut(\cE)(k)].
\end{equation*}
\end{exam}

\sss{Bounded version}
Let $\un\l=(\l_{i})_{i\in I}\in (\xcoch(T)^{+})^{I}$. We define a bounded version $\Sht^{(I_{1},\cdots, I_{r}), \le \un\l}_{G,\bK}$ of $\Sht_{G,\bK}^{(I_{1},\cdots, I_{r})}$ using the Cartesian diagram
\begin{equation*}
\xymatrix{\Sht_{G,\bK}^{(I_{1},\cdots, I_{r}),\le\un\l}\ar[r]\ar[d] & \Hk_{G,\bK}^{(I_{1},\cdots, I_{r}),\le\un\l}\ar[d]^{(p_{0},p_{r})}\\
\Bun_{G,\bK}\ar[r]^-{(\id,\Fr)} & \Bun_{G,\bK}\times \Bun_{G,\bK}
}
\end{equation*}
Then $\Sht^{(I_{1},\cdots, I_{r}), \le \un\l}_{G,\bK}$ is a closed substack of $\Sht_{G,\bK}^{(I_{1},\cdots, I_{r})}$.

\begin{remark} Our notation differs from that of \cite[Definition 2.1]{VLaff}. The stack $\Sht^{(I_{1},\cdots, I_{r})}_{G,\bK}$ is denoted $\textup{Cht}_{\bK,I}^{(I_{1},\cdots, I_{r})}$ in {\em loc.cit}. For the bounded version, the stack $\textup{Cht}_{\bK,I, \un\l}^{(I_{1},\cdots, I_{r})}$ was introduced in {
\em loc. cit}, and there is a closed embedding $\Sht^{(I_{1},\cdots, I_{r}), \le \un\l}_{G,\bK}\incl \textup{Cht}_{\bK,I, \un\l}^{(I_{1},\cdots, I_{r})}$ that induces an isomorphism after taking reduced structures.
\end{remark}

\begin{exam}\label{ex:Gm} Let $G=\GL_{1}=\Gm$ and $I$ be a finite set . A bound in this case is given by a sequence of integers $\un\l:=(\l_{i})_{i\in I}\in \ZZ^{I}$. Note that the partial order on $\ZZ^{I}$ is trivial: $\l'\le \l$ if and only if $\l'=\l$. The stack $\Sht^{I, \le \un\l}_{\Gm}$ admits an alternative description: it fits into a Cartesian diagram
\begin{equation*}
\xymatrix{\Sht^{I, \le \un\l}_{\Gm} \ar[d]\ar[r] & \Pic_{X}\ar[d]^{\bL_{\Pic_{X}}}\\
X^{I}\ar[r]^{ \AJ_{\un\l}} &  \Pic_{X} 
}
\end{equation*}
Here $\bL_{\Pic_{X}}$ is the Lang isogeny for $\Pic_{X}$ sending a line bundle $\cL$ to ${}^{\t}\cL\ot \cL^{-1}$. The Abel-Jacobi map $\AJ_{\un\l}$ sends $(x_{i})_{i\in I}$ to the line bundle $\cO_{X}(\sum_{i\in I}\l_{i}x_{i})$.

Since $\bL_{\Pic_{X}}$ has image in $\Pic^{0}_{X}$ while $\AJ_{\un\l}$ lands in $\Pic^{|\un\l|}_{X}$ where $|\un\l|=\sum_{i\in I}\l_{i}$, we see that $\Sht^{I, \le \un\l}_{G} =\vn$ if $|\un\l|\ne0$.
\end{exam}

\begin{exam}\label{ex:Dr} Let $G=\GL_{n}$, and let $D=\sum_{x\in \Sig}m_{x}x$ be an effective divisor on $X$. Let $\bK_{D}$ be the collection of principal congruence subgroups $\bK_{x,m_{x}}$ for $x\in \Sig$. Let $I=\{1,2\}$ with the decomposition into $I_{1}=\{1\}$ and $I_{2}=\{2\}$. Let $\l_{1}=(0,\cdots, 0,-1)$ and $\l_{2}=(1,0,\cdots,0)$ be two minuscule coweights, and $\un\l=(\l_{1},\l_{2})$. Then $\Sht^{(I_{1},I_{2}), \le\un\l}_{G,\bK_{D}}(S)$ classifies the following data
\begin{equation*}
\xymatrix{\cV_{0} & \ar[l]_-{\a_{1}}\cV_{1}  \ar[r]^-{\a_{2}} & \cV_{2} \ar[r]^{\io}_{\sim} &  {}^{\t}\cV_{0}}
\end{equation*}
where $\cV_{i}$ are vector bundles of rank $n$ over $X\times S$ with trivializations along $D\times S$, $\a_{1}$ and $\a_{2}$ are injective maps of coherent sheaves compatible with the trivializations along $D\times S$, whose cokernels are locally free over graphs of maps $x_{1}: S\to U=X-\Sig$ and $x_{2}: S\to U$ respectively.  This is the kind of moduli stack originally considered by Drinfeld in \cite{Dr-FBun}.
\end{exam}

\sss{Admissibility}
Let $\Om=\xcoch(T)/\ZZ\Phi^{\vee}$ where $\ZZ\Phi^{\vee}$ is the coroot lattice. Then $\Om$ is the group of connected components of $\Gr_{G}$ (and of $\QGr_{G}$). The set of connected components $\pi_{0}(\Bun_{G})$ can also be canonically identified with $\Om$, by \cite[Lemma 2.2]{Var}. More generally, if each $\bK_{x}$ is contained in a parahoric subgroup of $L_{x}G$, then $\pi_{0}(\Bun_{G,\bK})$ can again be canonically identified with $\Om$, with the component of the trivial bundle  corresponding to $0\in \Om$.

Let $\un\l=(\l_{i})_{i\in I}\in (\xcoch(T)^{+})^{I}$, and $\om$ be the image of $\sum_{i\in I}\l_{i}\in \xcoch(T)$ in $\Om$. We call $\un\l$ {\em $G$-admissible} if $\om=0\in \Om$ (see \cite[Definition 2.14]{Var}).

\begin{lemma}[Varshavsky {\cite[Proposition 2.16(d)]{Var}}] Suppose $\bK_{x}$ is contained in a parahoric subgroup for each $x\in \Sig$. Let $I=I_{1}\coprod\cdots\coprod I_{r}$ be a decomposition of a finite set $I$, and $\un\l=(\l_{i})_{i\in I}\in (\xcoch(T)^{+})^{I}$.

Then $\Sht^{(I_{1},\cdots, I_{r}), \le\un\l}_{G,\bK}$ is non-empty if and only if $\un\l$ is $G$-admissible.
\end{lemma}
\begin{proof}[Sketch of proof]

For $\g\in \Om\cong \pi_{0}(\Bun_{G,\bK})$, let $\Bun_{G, \bK}^{\g}$ be the corresponding  connected component of $\Bun_{G,\bK}$. Restricting the maps $(p_{0},p_{r}): \Hk^{(I_{1},\cdots, I_{r})}_{G,\bK}\to \Bun_{G,\bK}\times\Bun_{G,\bK}$ to $\Hk^{(I_{1},\cdots, I_{r}), \le\un\l}_{G,\bK}$, we have
\begin{equation*}
(p^{\le\un\l}_{0},p^{\le\un\l}_{r}):  \Hk^{(I_{1},\cdots, I_{r}), \le\un\l}_{G,\bK}\to \coprod_{\g\in \Om} \Bun_{G,\bK}^{\g}\times \Bun^{\g+\om}_{G,\bK} 
\end{equation*}
where $\om$ is the image of $\sum\l_{i}$ in $\Om$. 

If $\om\ne0$, the image of $(p^{\le\un\l}_{0},p^{\le\un\l}_{r})$ does not intersect $\Bun_{G,\bK}^{\g}\times \Bun^{\g}_{G,\bK}$ for any $\g\in \Om$. However, the graph of Frobenius on $\Bun_{G,\bK}$ lies over  $\coprod_{\g\in \Om} \Bun_{G,\bK}^{\g}\times \Bun^{\g}_{G,\bK} 
$. By the Cartesian diagram \eqref{defn Sht}, we see $\Sht^{(I_{1},\cdots, I_{r}), \le\un\l}_{G,\bK}=\vn$.

Conversely, assume $\un\l$ is $G$-admissible. Since $\Hk^{(I_{1},\cdots, I_{r}),\le\un\l}_{G, \bK}\to \Hk^{I,\le\un\l}_{G,\bK}$ is surjective, so is the map $\Sht^{(I_{1},\cdots, I_{r}),\un\l}_{G, \bK}\to \Sht^{I,\le\un\l}_{G,\bK}$. Therefore it suffices to show that $\Sht^{I,\le\un\l}_{G,\bK}\ne\vn$. 

Let $\l=\sum_{i}\l_{i}$. Since $\un\l$ is admissible, the dominant coweight $\l$ lies in the coroot lattice of $G$, hence $0\le \l$. Restricting $\Sht^{I,\le\un\l}_{G,\bK}$ to the diagonal $\D(U)\subset U^{I}$, we have $\Sht^{I,\le\un\l}_{G,\bK}|_{\D(U)}\cong \Sht^{\{1\},\le\l}_{G,\bK}$. This will be discussed in \S\ref{ss:geom fact}. Therefore we only need to show $\Sht^{\{1\},\le\l}_{G,\bK}\ne\vn$. Since $0\le \l$, we have $\Sht^{\{1\},\le0}_{G,\bK}\subset \Sht^{\{1\},\le\l}_{G,\bK}$. Clearly, $\Sht^{\{1\},\le0}_{G,\bK}=\Sht^{\vn}_{G,\bK}\times U$, which is non-empty by Example \ref{ex:Sht no leg}.
\end{proof}

\subsection{Iso-Shtukas}

Recall $F$ is the function field of $X$.

\sss{The tensor category of iso-Shtukas}
Let $\IsoSht(F)$ be the category of pairs $(V, \phi)$ where
\begin{itemize}
\item $V$ is a finite-dimensional $\br F:=F\ot_{k} \ov k$-vector space.
\item $\phi: V\isom V$ is an $F$-linear and $(\ov k, \Fr)$-linear isomorphism. I.e., it is bijective and for $v\in V$, $c\in \ov k$ and $a\in F$, we have 
$\phi(cv)=c^{q}\phi(v)$ and $\phi(av)=a\phi(v)$. 
\end{itemize}
We call objects in $\IsoSht(F)$ {\em iso-Shtukas} for the function field $F$. The rank of $(V,\phi)\in \IsoSht(F)$ is defined to be $\dim_{\br F}V$. Then $\IsoSht(F)$ is a rigid $F$-linear tensor category under the operation $(V_{1},\phi_{1})\ot (V_{2}, \phi_{2})=(V_{1}\ot_{\br F} V_{2}, \phi_{1}\ot\phi_{2})$. The dual object of  $(V,\phi)$ is $(V^{\vee}=\Hom_{\br F}(V,\br F), \phi^{\vee})$ where $\phi^{\vee}$ is defined by $\j{\phi^{\vee}(\xi),\phi(v)}=\j{\xi, v}$ for $\xi\in V^{\vee}, v\in V$.

We should think of iso-Shtukas over $F$ as Shtukas over the generic point of $X$. More precisely, to any $\ov k$-point $((x_{i})_{i\in I}, \cE, \a)\in \Sht^{I}_{\GL_{n},\bK}(\ov k)$, there corresponds an iso-Shtuka defined as follows. Let $V$ be the  fiber of $\cE$ over the generic point $\y=\Spec \br F$ of $X_{\ov k}$. Then $\a$ restricted to $\y$ gives an $\br F$-linear isomorphism $\a: V\isom V\ot_{\ov k,\Fr}\ov k$. We define $\phi(v)=\a^{-1}(v\ot 1)$ for $v\in V$.  Then $(V,\phi)$ is an object in $\IsoSht(F)$ of rank $n$.

For a reductive group $G$ over $F$, we define a {\em $G$-iso-Shtuka} to be an $F$-linear tensor functor
\begin{equation}\label{GisoSh}
\Rep_{F}(G)\to \IsoSht(F).
\end{equation}
If $G$ is split, a $\ov k$-point $((x_{i})_{i\in I}, \cE, \a)\in \Sht^{I}_{G,\bK}(\ov k)$ gives rise to a $G$-iso-Shtuka as follows: it sends a finite-dimensional $F$-representation $W$ of $G$ to the iso-Shtuka $(W_{\cE,\y}, \phi_{W})$ where $W_{\cE}=(\cE|_{U_{\ov k}})\times^{G}W$ is the vector bundle over $U_{\ov k}$ associated to $\cE$ and $W$, and $W_{\cE,\y}$ is its generic fiber; and $\phi_{W}$ is the inverse of the restriction of $\a_{W}: W_{\cE}\dr {}^{\t}(W_{\cE})\cong W_{{}^{\t}\cE}$ at the generic point $\y$ of $X_{\ov k}$.  

\sss{Isocrystals} It is clear that iso-Shtukas are global analogs of isocrystals. Let $x\in |X|$. Let $\br F_{x}:=\wh{F_{x}^{\ur}}$ be the completion of the maximal unramified extension of $F_{x}$. Let $\Fr_{x}\in\Aut(\br F_{x})$ be the lifting of the automorphism $a\mapsto a^{\# k_{x}}$ of $k_{x}$. Recall that the category $\IsoCrys(F_{x})$ of isocrystals over $\br F_{x}$ with respect to $\Fr_{x}$ consists of pairs $(V, \phi)$ where 
\begin{itemize}
\item $V$ is a finite-dimensional $\br F_{x}$-vector space.
\item $\phi: V\isom V$ is a $(\br F_{x}, \Fr_{x})$-linear isomorphism. I.e., it is bijective, and for $v\in V$,  $a\in \br F_{x}$, we have 
$\phi(av)=\Fr_{x}(a)\phi(v)$.
\end{itemize}
Then $\IsoCrys(F_{x})$ is an $F_{x}$-linear rigid tensor category. 

Given an iso-Shtuka $(V,\phi)$ and a place $x\in|X|$, we may form an isocrystal over $\br F_{x}$ by extension of scalars
\begin{equation*}
(V,\phi)\mapsto (V\ot_{\br F}\br F_{x}, \phi^{d_{x}}\ot \Fr_{x}).
\end{equation*}
Here $d_{x}=[k_{x}:k]$. The above construction gives a $F$-linear tensor functor
\begin{equation}\label{Sht fib functor}
\om_{x}: \IsoSht(F)\to \IsoCrys(F_{x}).
\end{equation}

\sss{Classification of simple objects}
Recall the Dieudonn\'e-Manin classification for isocrystals: $\IsoCrys(F_{x})$ is a semisimple $F_{x}$-linear tensor category; the simple objects are classified by slopes $\l\in \QQ$. The simple object $M(\l)\in \IsoCrys(F_{x})$ of slope $\l=m/n$ ($n\ge1$, $m\in \ZZ$ prime to $n$) has dimension $n$ over $\br F_{x}$. For $\l_{1}, \l_{2}\in \QQ$, $M(\l_{1})\ot M(\l_{2})$ is isomorphic to a direct sum of $M(\l_{1}+\l_{2})$. The dual of $M(\l)$ is isomorphic to $M(-\l)$. The endomorphism algebra of $M(\l)$ is a central division algebra over $F_{x}$ with invariant $-\l\mod\ZZ\in \QQ/\ZZ$.

Drinfeld \cite[\S2]{Dr-Pet} describes the structure of $\IsoSht(F)$ as an abelian category. To state it, let $F^{\sep}$ be a separable closure of $F$, and consider the $\QQ$-vector space $F^{\sep,\times}\ot_{\ZZ}\QQ$ (usually simplified as $F^{\sep,\times}\ot\QQ$) with the action of $\G_{F}:=\Gal(F^{\sep}/F)$. To each element  $a\in F^{\sep,\times}\ot\QQ$, there is a unique smallest finite extension $L_{a}/F$ inside $F^{\sep}$ such that $a$ is in the image of $L^{\times}_{a}\ot\QQ\to F^{\sep,\times }\ot\QQ$.  

\begin{theorem}[Drinfeld \cite{Dr-Pet}]\label{th:isoSht}
\begin{enumerate}
\item The category $\IsoSht(F)$ is semisimple.
\item Isomorphism classes of simple objects in $\IsoSht(F)$ are in natural bijection with $\G_{F}$-orbits on $F^{\sep,\times}\ot\QQ$.
\item Let $(V,\phi)$ be a simple object in $\IsoSht(F)$ with isomorphism class corresponds to the $\G_{F}$-orbit of $a\in F^{\sep,\times}\ot\QQ$. Define $L_{a}$ as above. Write $\div(a)=\sum_{y\in |L_{a}|}\ord_{y}(a)\cdot y$ (sum over places of $L_{a}$). Then 
\begin{enumerate}
\item $\End(V,\phi)$ is a central division algebra over $L_{a}$ with invariant $-\ord_{y}(a)[k_{y}:k]\mod \ZZ\in \QQ/\ZZ$ at $y\in |L_{a}|$. 
\item $\dim_{\br F}V=[L_{a}:F]d_{a}$, where $d_{a}\in \NN$ is the least common denominator of the set of rational numbers $\{\ord_{y}(a)[k_{y}:k]\}_{y\in |L_{a}|}$.
\item Since $L_{a}\subset \End(V,\phi)$, we may view $(V,\phi)$ as an object in $\IsoSht(L_{a})$, and it makes sense to localize $(V,\phi)$ at a place  $y\in |L_{a}|$ to obtain $\om_{y}(V,\phi)\in \IsoCrys(L_{a,y})$. Then $\om_{y}(V,\phi)$ is a direct sum of simple isocrystals of slope $\ord_{y}(a)$.
\end{enumerate}

\end{enumerate}
\end{theorem}

We give the construction in one direction of part (2) in the Theorem. From a simple object $(V,\phi)\in \IsoSht(F)$, we obtain $a\in (F^{\sep,\times}\ot\QQ)/\G_{F}$ as follows. There exists a finite extension $k_{n}/k$ and a descent $V_{n}$ of   $V$ over $F\ot_{k}k_{n}$ such that $\phi$ restricts to a $\id_{F}\ot \Fr_{k_{n}/k}$-linear map $\phi_{n}: V_{n}\to V_{n}$. We define $a$ to be the Galois orbit of $\l^{1/n}$, where $\l\in F^{\sep,\times}$ is any eigenvalue of the $F\ot_{k}k_{n}$-linear map $(\phi_{n})^{n}$. Changing the choices of $n$, the descent $V_{n}$ and the eigenvalue $\l$ of $(\phi_{n})^{n}$, the element $\l^{1/n}\in F^{\sep, \times}$ only gets multiplied by a root of unity acted by $\G_{F}$. We thus get a well-defined $\G_{F}$-orbit in $F^{\sep,\times}\ot\QQ$.

\sss{The tensor structure} To describe the tensor structure on $\IsoSht(F)$, we consider the following function
\begin{equation*}
\chi: K_{0}(\IsoSht(F))\to \ZZ[F^{\sep, \times}\ot\QQ]^{\G_{F}}
\end{equation*}
sending a simple object $(V,\phi)$ corresponding to the Galois orbit of $a\in F^{\sep,\times}\ot\QQ$ to the function $\chi(V,\phi): F^{\sep, \times}\ot\QQ\to \ZZ$ that is supported on the $\G_{F}$-orbit of $a$, constant there, and with total sum equal to $\dim_{\br F} V$. 

The following proposition describes how simple objects in $\IsoSht(F)$ tensors with each other.
\begin{prop}\label{p:IsoSht tensor} The map $\chi$ is an injective ring homomorphism.
\end{prop}
This can be shown by computing the local invariants of the semisimple $F$-algebra $\End(V_{1}\ot V_{2})\cong\End(V_{1})\ot_{ F}\End(V_{2})$ for two simple objects $V_{1},V_{2}\in \IsoSht(F)$.

The forgetful functor gives  a fiber functor $\om: \IsoSht(F)\to \Vect(\br F)$, hence $\IsoSht(F)$ is a Tannakian category. Its Tannakian group over $\br F$ is a diagonalizable group $\DD_{\br F}$ with character group $\xch(\DD_{\br F})=F^{\sep,\times}\ot\QQ$, equipped with the action of $\Gal(F^{\sep}/\br F)$.


\section{Siblings and cousins}\label{s:sib}
In this section, we mention several constructions that are closely related to Shtukas. We can think of Drinfeld modules as an identical twin of Shtukas, the Deligne-Lusztig varieties as a baby sibling, while motives over a finite field as a cousin. Due to limitation of the author's expertise, we have omitted an important mixed-characteristic cousin of Shtukas discovered in the past decade by Scholze \cite{S-Berkeley}. 

\subsection{Drinfeld modules} 

We recall the notion of Drinfeld modules and the dictionary between Drinfeld modules and Shtukas of a specific kind.

\sss{Drinfeld modules}
Let $\infty\in |X|$.  Let $A=\G(X-\{\infty\}, \cO_{X})$, a Dedekind domain over $k$.

We recall the notion of a Drinfeld $A$-module of rank $n$ (or an elliptic module, according to \cite{Dr-Ell}) over a test $k$-scheme $S$. It is a pair $(\cG, \io)$ where  $\cG$ is a one-dimensional additive group scheme over $S$ (Zariski locally over $S$ isomorphic to $\GG_{a,S}$), and $\io: A\to \End_{S}(\cG)$ is a $k$-algebra homomorphism such that for every geometric point $s=\Spec K \to  S$, and any isomorphism of $K$-group schemes $\cG_{s}\cong\GG_{a,K}$, the restriction of $\io$ to the $s$-fiber
\begin{equation*}
\io_{s}: A\to \End_{K}(\cG_{s})\cong K\j{\t}
\end{equation*}
(where $\t\in \End_{K}(\cG_{s})\cong\End_{K}(\GG_{a,K})$ is the Frobenius morphism $x\mapsto x^{q}$) takes the form
\begin{equation*}
\io_{s}(a)=\sum_{i=0}^{-n[k_{\infty}:k]\ord_{\infty}(a)}c_{i}(a)\t^{i}, \quad \forall a\in A-\{0\}
\end{equation*}
and $c_{-n[k_{\infty}:k]\ord_{\infty}(a)}(a)\ne0$ for $a\ne0$.

The action of $A$ on the Lie algebra of $\cG$ gives a ring homomorphism $A\to \G(S,\cO_{S})=\End_{S}(\Lie\cG)$, or equivalently a map $\xi: S\to \Spec A=X-\{\infty\}$. This is called the {\em characteristic} of the Drinfeld $A$-module $(\cG,\io)$.

Let $\Dr_{n}(S)$ be the groupoid of Drinfeld $A$-module of rank $n$ over $S$.

\sss{Drinfeld modular variety}\label{sss:one leg}
To relate Drinfeld modules to Shtukas, we consider a variant of the moduli problem in Example \ref{ex:Dr}.

For simplicity we assume the residue field of $\infty$ is $k$. Let $\Sht_{n}^{\Dr}$ be the moduli stack whose $S$ points are $(x, \cE, \{\cE(\frac{i}{n})\}_{i\in \ZZ}, \phi)$ where
\begin{enumerate}
\item $x: S\to U:=X-\{\infty\}$.
\item $\cE$ is a rank $n$ vector bundle over $X\times S$ such that for any geometric point $s\in S$, the coherent Euler characteristic $\chi(X_{s}, \cE_{s})=0$ (i.e., $\dim \upH^{0}(X_{s},\cE_{s})=\dim\upH^{1}(X_{s},\cE_{s})$).
\item $\cE(-1)=\cE\ot\cO_{X}(-\infty)\subset \cE(-\frac{n-1}{n})\subset\cdots\subset \cE(-\frac{1}{n})\subset \cE(0)=\cE$ is a full flag of $\cE|_{\{\infty\}\times S}$, i.e., $\cE(-\frac{i}{n})/\cE(-\frac{i+1}{n})$ is locally free of rank one over $\{\infty\}\times S$, for $i=0,1,\cdots, n-1$.  We extend the definition of $\cE(\frac{i}{n})$ to all $i\in \ZZ$ by requiring $\cE(1+\frac{i}{n})=\cE(\frac{i}{n})\ot\cO_{X}(\infty)$.
\item $\phi: {}^{\t}\cE\to \cE(\frac{1}{n})$ with cokernel locally free of rank one over $\G(x)$, and $\phi$ sends ${}^{\t}(\cE(-\frac{i}{n}))$ to $\cE(-\frac{i-1}{n})$ for all $i=1,\cdots,n-1$.
\end{enumerate}


Let $I=\{1,2\}$ with the decomposition into singletons $I_{1}=\{1\}$ and $I_{2}=\{2\}$. Let $\l_{1}=(0,\cdots,0,-1)$ and $\l_{2}=(1,0,\cdots ,0)$ be two minuscule coweights of $\GL_{n}$. In Example \ref{ex:Dr} we defined the moduli stack $\pi: \Sht^{(I_{1},I_{2}), (\l_{1},\l_{2})}_{\GL_{n}} \to X^{2}$. Then there is a map 
\begin{equation}\label{Dr to Sht}
\Sht_{n}^{\Dr}\to \Sht^{(I_{1},I_{2}), (\l_{1},\l_{2})}_{\GL_{n}}|_{U\times\{\infty\}}
\end{equation}
sending $(x, \cE, \{\cE(\frac{i}{n})\}_{i\in \ZZ}, \phi)$ to the two legs $(x, \infty)$ and bundles
\begin{equation*}
\xymatrix{\cE_{0}=\cE & \cE_{1}:={}^{\t}(\cE(\frac{-1}{n}))\ar@{_{(}->}[l]_-{\phi}\ar@{^{(}->}[r] & \cE_{2}:={}^{\t}\cE}
\end{equation*}
It can be shown that \eqref{Dr to Sht} is a closed embedding; the $\ov k$-points in its image consists of Shtukas whose isocrystals at $\infty$ are supersingular with slope $1/n$.

\sss{Dictionary}\label{sss:dict} According to \cite{Mum}, there is a canonical equivalence of groupoids
\begin{equation*}
\Dr_{n}(S)\cong \Sht^{\Dr}_{n}(S). 
\end{equation*}



We describe the functors in both directions. More details can be found in \cite[Chapter 6]{Goss}. Below we assume $S$ is locally noetherian.

Drinfeld modules to Shtukas. Let $(\cG,\io)$ be a Drinfeld module of rank $n$ over $S$. Then $M=\un\Hom_{S}(\cG,\GG_{a,S})$ is a quasi-coherent sheaf on $S$ that is  locally free of rank one over $\cO_{S}\j{\t}=\un\End_{S}(\GG_{a,S})$. Moreover, $M$ has an action of $A$ via $\io$  commuting with the $\cO_{S}\j{\t}$-action. We will define a canonical increasing filtration $M_{i}$ on $M$ by $\cO_{S}$-submodules as follows. Locally on $S$,  let $m_{0}\in M$ be a basis of $M$ as a free $\cO_{S}\j{\t}$-module. Let $M'(\frac{i}{n})$ be the $\cO_{S}$-submodule of $M$ generated by $m_{0}, \t m_{0},\cdots, \t^{i-1}m_{0}$. We then renormalize the filtration $\{M'(\frac{i}{n})\}$: define $M(\frac{i}{n})$ to be those $m\in M$ such that $am\in M'(\frac{i}{n}+\ord_{\infty}(a))$ for $\ord_{\infty}(a)$ sufficiently large (depending on $m$). Two local choices of $m_{0}$ define the same $M'(\frac{i}{n})$ for $i\gg0$ (because any other basis $m_{1}$ of $M$ as an $\cO_{S}\j{\t}$-module takes the form $m_{1}=(c_{0}+c_{1}\t+\cdots +c_{N}\t^{N})m_{0}$ for $c_{1}\in \cO_{S}^{\times}$ and $c_{i}$ nilpotent for $i\ge1$). This implies that the filtration $\{M(\frac{i}{n})\}_{i\ge1}$ is independent of the choice of the local basis $m_{0}$, hence defining a canonical filtration of $M$ by coherent $\cO_{S}$-submodules. From the construction we see that $\t M(\frac{i}{n})\subset M(\frac{i+1}{n})$. Note $A$ is filtered by the pole order at $\infty$: $A=\cup_{j\ge0}A_{j}$. Let $\Rees(A)=\op_{j\ge0}A_{j}$, then $X=\Proj (\Rees(A))$. By construction, $A_{j}$ sends $M(\frac{i}{n})$ to $M(\frac{i}{n}+j)$.  Let $N(\frac{i}{n})$ be the graded $\Rees(A)\ot\cO_{S}$-module whose degree $j$ piece is $M(\frac{i}{n}+j)$. Let $\cE(\frac{i}{n})$ be the coherent sheaf on $X\times S=\Proj_{S}(\Rees(A)\ot\cO_{S})$ associated with $N(\frac{i}{n})$. One checks that $\cE(\frac{i}{n})$ are vector bundles of rank $n$ over $X\times S$, and the graded embeddings $N(\frac{i}{n})\incl N(\frac{i+1}{n})$ induce injective maps $\cE(\frac{i}{n})\incl \cE(\frac{i+1}{n})$ whose cokernel is rank one on $\{\infty\}\times S$. Finally, the map $\phi: {}^{\t}\cE={}^{\t}\cE(0)\to \cE(\frac{1}{n})$ comes from the $\Fr_{S}$-linear map $\phi_{N}: N(0)\to N(\frac{1}{n})$, which in degree $j$ is $\t: M(j)\to M(\frac{1}{n}+j)$.

Shtukas to Drinfeld modules. Let $(x, \cE, \{\cE(\frac{i}{n})\}_{i\in \ZZ}, \phi)\in \Sht^{\Dr}(S)$. Let $\pr_{S}: X\times S\to S$ be the projection. Let $M$ be the quasi-coherent $\cO_{S}$-module $M=\varinjlim_{i}\pr_{S*}\cE(\frac{i}{n})$ with a commuting action of $A$. The map $\phi$ gives a $\Fr_{S}$-linear map $\phi_{M}: M\to M$. One shows that $M$ is a locally free of rank one over the skew-polynomial ring $\cO_{S}\j{\t}$ over $\cO_{S}$, where $\t$ acts by $\phi_{M}$.  Since $\un\End_{S}(\GG_{a,S})=\cO_{S}\j{\t}$, a locally free rank one $\cO_{S}\j{\t}$-module is the same as a Zariski locally trivial torsor for $\un\Aut(\GG_{a,S})\cong \cO_{S}\j{\t}^{\times}$, which in turn is the same as a Zariski locally trivial form $\cG$ of $\GG_{a}$ over $S$.  Since $M$ has an $A$-action commuting with the $\cO_{S}\j{\t}$-action, $\cG$ also has an $A$-action, which gives the data of a Drinfeld module.

\sss{$t$-motives} G.Anderson's theory of $t$-motives \cite{And} is a generalization of Drinfeld modules for $A=k[t]$. On the one hand, if we think of Drinfeld modules as function field analogs of elliptic curves, Anderson's $t$-motives are analogs of more general motives  including analogs of higher dimensional abelian varieties.  On the other hand, the dictionary in \S\ref{sss:dict} can be generalized to $t$-motives: $t$-motives  can be turned into Shtukas on $X=\PP^{1}$ with a fixed leg at $\infty$ and another moving leg. For details we refer to Anderson's original paper \cite{And} and the book \cite{Goss} by D. Goss.

\subsection{(Affine) Deligne-Lusztig varieties}
Let $\cB$ be the flag variety of $G$: it is the moduli space of Borel subgroups of $G$.  The set of $G$-orbits on $\cB\times\cB$ (under the diagonal action) form a Coxeter group $(W,S)$ and is isomorphic to the Weyl group attached to a maximal split torus of $G$.  For $(B,B')\in \cB^{2}$, we write $\pos(B,B')=w$ if $(B,B')$ lies in the $G$-orbit indexed by $w\in W$; we write $\pos(B,B')\le w$ if $(B,B')$ lies in the closure of the $G$-orbit indexed by $w\in W$.

We may define an analogue of $\Hk^{(I_{1},\cdots, I_{r})}_{G}$ (at least for each $I_{j}$ being a singleton set) by replacing $G$-torsors with Borel subgroups. Let $\un w=(w_{j})_{1\le j\le r}\in W^{r}$. Define $Y^{\un w}$ to be the subscheme of $\cB^{r+1}$ consisting of Borel subgroups $(B_{0},\cdots, B_{r})$ of $G$ such that $\pos(B_{j-1},B_{j})=w_{j}$ for $1\le j\le r$. Similarly we define $Y^{\le \un w}\subset \cB^{r+1}$ be the closed subscheme where  $\pos(B_{j-1},B_{j})\le w_{j}$ for $1\le j\le r$.

We may then define an analogue of $\Sht^{(I_{1},\cdots, I_{r})}_{G}$. Define $X^{\un w}$ and $X^{\le \un w}$ by the Cartesian diagrams
\begin{equation*}
\xymatrix{ X^{\un w}\ar[r]\ar[d] & Y^{\un w}\ar[d]^{(p_{0}, p_{r})} & X^{\le\un w}\ar[r]\ar[d]  & Y^{\le\un w}\ar[d]^{(p_{0}, p_{r})} \\
\cB\ar[r]^-{(\id,\Fr)} & \cB^{2} & \cB\ar[r]^-{(\id,\Fr)} & \cB^{2}
}
\end{equation*}
Here, $p_{j}$ is the $j$-th projection $\cB^{r+1}\to \cB$. Note that when $r=1$, $X^{w}$ is the Deligne-Lusztig variety \cite[Definition 1.4]{DL} that classifies Borel subgroups $B$ such that $\pos(B,\Fr(B))=w$. The Deligne-Lusztig varieties play a vital role in the construction and classification of irreducible representations of finite groups of Lie type.

Similarly, if we replace $G$ by the loop group $LG$, $B$ by the Iwahori $\bI\subset LG$, $\cB$ by the affine flag variety $\Fl_{G}=LG/\bI$, $W$ by the extended affine Weyl group $\tilW$, we may similarly define $\cY^{\un w}\subset \cY^{\le\un w}\subset (\Fl_{G})^{r+1}$ for $\un w\in (\tilW)^{r}$. We can then  intersect them with the graph of Frobenius on $\Fl_{G}$ to obtain special cases of the {\em affine Deligne-Lusztig varieties} $\cX^{\un w}$ and $\cX^{\le \un  w}$.

More general affine Deligne-Lusztig varieties are defined by Rapoport \cite[\S4]{Ra}. To define them, one fixes $b\in LG$ and replaces the Frobenius map $\Fl_{G}\to \Fl_{G}$ by $b\c\Fr$ in the previous paragraph, where $b$ acts by left translation. The resulting sub-ind-schemes of $\Fl_{G}$ are denoted $\cX^{\un w}(b)$ and $\cX^{\le\un w}(b)$, whose isomorphism classes only depend on the $\Fr$-conjugacy class of $b$.

\subsection{Motives over a finite field}
Now we would like to point out the similarities between Shtukas and motives over finite fields. The main references are Milne \cite{Milne} and Langlands-Rapoport \cite{LR}. 

Fix a prime $p$. For each finite extension $\FF_{q}$ of $\FF_{p}$ there is a $\QQ$-linear tensor category $\Mot(\FF_{q})$ of (pure) motives over $\FF_{q}$, whose hom spaces are defined using cycles on the product up to numerical equivalence. Similarly we have the $\QQ$-linear tensor category $\Mot(\ov\FF_{p})$ of (pure) motives over $\ov\FF_{p}$. It is the colimit of $\Mot(\FF_{q})$ as $q$ runs over finite extensions of $\FF_{p}$. These categories are known to be semisimple \cite{Jan}.

We have the following description of $\Mot(\ov\FF_{p})$ as an abelian category, generalizing the Honda-Tate theory for abelian varieties over finite fields. Recall a Weil $p$-number of weight $n\in\ZZ$ is an algebraic number $\pi$ that is an $\ell$-adic unit for all primes $\ell\ne p$, and all archimedean norms of $\pi$ are equal to $p^{n/2}$. Let $W_{p}\subset \ov \QQ$ be the set of Weil $p$-numbers of various integer weights. Let $\mu_{\infty}\subset \ov\QQ^{\times}$ be the roots of unity. It is permuted by the Galois group $\G_{\QQ}=\Gal(\ov\QQ/\QQ)$. Given $\pi\in W_{p}/\mu_{\infty}$, define a number field $L_{\pi}=\cap_{n\in\NN}\QQ[\pi^{n}]$.

The following theorem should be compared with Theorem \ref{th:isoSht}.
\begin{theorem}[See {\cite[Theorem 2.18, 2.19]{Milne}}]\label{th:mot}
\begin{enumerate}
\item The isomorphisms classes of simple objects in $\Mot(\ov\FF_{p})$ are in natural bijection with $\G_{\QQ}$-orbits on $W_{p}/\mu_{\infty}$.
\item  If a simple motive $X\in \Mot(\ov\FF_{p})$ corresponds to the Galois orbit of $\pi\in W_{p}/\mu_{\infty}$ of weight $w$, then 
\begin{enumerate}
\item $\End(X)$ is a central division algebra over the number field $L_{\pi}$ whose invariant at a place $v\in |L_{\pi}|$ is given by
\begin{equation*}
\inv_{v}(\End(X))=\begin{cases}  -\ord_{v}(\pi)[k_{v}:\FF_{p}], & v\mid p \\ \frac{1}{2}, &  v\mbox{ real, $w$ is odd} ; \\ 0, & \textup{ otherwise. }\end{cases}
\end{equation*}

\item The rank of $X$ is $[L_{\pi}:\QQ]d_{\pi}$, where $d_{\pi}$ is the least common denominator of the collection of rational numbers $\{\ord_{v}(\pi)[k_{v}:\FF_{p}]\}_{v|p}$.
\end{enumerate}
\end{enumerate}
\end{theorem}

To describe the tensor structure on $\Mot(\ov\FF_{p})$,  we consider the following linear function
\begin{equation*}
\chi: K_{0}(\Mot(\ov\FF_{p}))\to \ZZ[W_{p}/\mu_{\infty}]^{\G_{\QQ}}
\end{equation*}
sending a simple motive $X$ corresponding to the Galois orbit of $\pi\in W_{p}/\mu_{\infty}$ to the function $\chi(X): W_{p}/\mu_{\infty}\to \ZZ$ that is supported on the $\G_{\QQ}$-orbit of $\pi$ with constant value $d_{\pi}$.  The following is analogous to Prop. \ref{p:IsoSht tensor}:

\begin{prop}[See {\cite[Corollary 2.20]{Milne}}] The map $\chi$ is an injective ring homomorphism.
\end{prop}

\sss{Fiber functors}
For primes $\ell\ne p$,  taking $\ell$-adic \'etale cohomology of a motive gives a fiber functor
\begin{equation*}
\om_{\ell}: \Mot(\ov\FF_{p})\ot\QQ_{\ell}\to \Vect(\Ql).
\end{equation*}
For $\ell=p$, taking the crystalline cohomology of a motive over $\ov\FF_{p}$ gives a tensor  functor 
\begin{equation*}
\om_{p}: \Mot(\ov\FF_{p})\ot\QQ_{p}\to \IsoCrys(\QQ_{p}).
\end{equation*}
Assuming the Tate conjecture, $\Mot(\ov\FF_{p})\ot\RR$ also admits a fiber functor into $\RR$-vector spaces. These fiber functors are analogues of the functors $\om_{x}$ on $\IsoSht(F)$ defined in \eqref{Sht fib functor}.

\sss{Iso-Shtukas for $\QQ$?}
Comparing Theorem \ref{th:mot} with Theorem \ref{th:isoSht}, it suggests that motives over $\ov\FF_{p}$ should be thought of as iso-Shtukas for the number field $\QQ$ with legs at $p$ and $\infty$. This analogy works especially well between the motives of elliptic curves over $\ov \FF_{p}$ and the Shtukas associated to Drinfeld modules of rank $2$.

In \cite{Kot}, Kottwitz defines a semisimple $\QQ$-linear abelian category as representations of a gerbe over $\QQ$ band by a pro-torus. Conjecturally it includes $\Mot(\ov\FF_{p})$ as a full subcategory for all primes $p$. In \cite{SchICM}, Scholze suggests that Kottwitz's category should be thought of as $\IsoSht(\QQ)$ (with an arbitrary finite number of legs).

\begin{ques} If motives over $\ov\FF_{p}$ give iso-Shtukas for $\QQ$ with legs at $p$ and $\infty$, what geometric objects give iso-Shtukas for $\QQ$ with legs at several primes?
\end{ques}

%
%
%

\section{Geometry of the moduli stack}\label{s:geom}
In this section, we summarize basic geometric properties and structures on the moduli stack of $G$-Shtukas. They are responsible for the symmetries on the cohomology of these moduli stacks that will be studied in the next section.

\subsection{Representability}
The most basic result on the geometry of the moduli of Shtukas is the following.
\begin{theorem}[Varshavsky {\cite[Proposition 2.16(a)]{Var}}] Let $I=I_{1}\coprod\cdots\coprod I_{r}$ be a decomposition of a finite set, and let $\un\l=(\l_{i})_{i\in I}\in (\xcoch(T)^{+})^{I}$. The stack $\Sht^{(I_{1},\cdots, I_{r}), \le \un\l}_{G,\bK}$ is a Deligne-Mumford stack locally of finite type over $k$.
\end{theorem}

It is well-known that $\Bun_{G,\bK}$ is a smooth Artin stack locally of finite type over $k$. Moreover the maps $p_{i}: \Hk_{G,\bK}^{(I_{1},\cdots, I_{r}), \le \un\l}\to \Bun_{G,\bK}$ are of finite type. In view of the diagram \eqref{defn Sht} defining $\Sht^{(I_{1},\cdots, I_{r}), \le\un\l}_{G,\bK}$, we see that it is an Artin stack locally of finite type over $k$. 

\sss{Harder-Narasimhan truncation}\label{sss:HN}
One can cover $\Sht^{(I_{1},\cdots, I_{r}), \le \un\l}_{G,\bK}$ by open substacks of finite type over $k$ as follows. 

We first consider the case without level structure. Let $\xcoch(T)^{+}_{\QQ}$ be the dominant cone in $\xcoch(T)_{\QQ}$. Let $\HN\subset \xcoch(T)^{+}_{\QQ}$ be the subset of $\mu\in \xcoch(T)_{\QQ}^{+}$ satisfying the following integrality condition:  if $L_{\mu}$ is the standard Levi subgroup with roots $\{\a|\j{\a,\mu}=0\}$, then $\mu\in \xcoch(L_{\mu,\ab})_{\QQ}=\xcoch(ZL_{\mu})_{\QQ}\subset\xcoch(T)_{\QQ}$ (here $L_{\mu,\ab}$ is the abelianization of $L_{\mu}$).   Each geometric point $\cE\in \Bun_{G}$ has a canonical $P$-reduction $\cE_{P}$ for a standard parabolic $P\subset G$ (generalizing the Harder-Narasimhan filtration for vector bundles) such that: let $L$ be the Levi quotient of $P$, then the induced $L_{\ab}$-torsor $\cE_{L_{\ab}}=\cE_{P}/(U_{P}[L,L])$ has degree $\deg(\cE_{L_{\ab}})\in \xcoch(L_{\ab})$ that satisfy $\j{\a, \deg(\cE_{L_{\ab}})}>0$ for all roots $\a$ in $U_{P}$. Therefore to each geometric point $\cE\in \Bun_{G}$ we have a well-defined Harder-Narasimhan point $\deg(\cE_{L_{\ab}})\in \HN$.

Consider the partial order on $\HN$ given by: $\mu\le \mu'$ if and only if $\mu'-\mu $ is a $\QQ_{\ge0}$-linear combination of simple coroots. There is an open substack ${}^{\le\mu}\Bun_{G}\subset \Bun_{G}$ whose geometric points have HN point $\le \mu$. It is known that ${}^{\le\mu}\Bun_{G}$ is of finite type over $k$.

It is easy to see that the forgetful map $p_{0}: \Sht^{(I_{1},\cdots, I_{r}), \le \un\l}_{G}\to \Bun_{G}$ recording $\cE_{0}$ is of finite type. Now for $\mu\in \HN$, let ${}^{\le \mu}\Sht^{(I_{1},\cdots, I_{r}), \le \un\l}_{G}=p_{0}^{-1}({}^{\le \mu}\Bun_{G})$, then ${}^{\le \mu}\Sht^{(I_{1},\cdots, I_{r}), \le \un\l}_{G}$ is of finite type over ${}^{\le \mu}\Bun_{G}$, hence of finite type over $k$. As $\mu\in \HN$ varies, $\{{}^{\le \mu}\Sht^{(I_{1},\cdots, I_{r}),\le\un\l}_{G}\}_{\mu\in\HN}$ cover $\Sht^{(I_{1},\cdots, I_{r}),\le\un\l}_{G}$.

For a general level structure $\bK=(\bK_{x})_{x\in \Sig}$, we may choose $n_{x}\in\NN$ such that $\bK_{x}$ contains a principal congruent subgroup $L^{(n_{x})}_{x}G$. Let $D=\sum_{x\in \Sig}n_{x}x$ and $\bK_{D}=(L^{(n_{x})}_{x}G)_{x\in \Sig}$, then $\Sht^{(I_{1},\cdots, I_{r})}_{G,\bK_{D}}\to \Sht^{(I_{1},\cdots, I_{r})}_{G,\bK}$ is a finite \'etale map (see \S\ref{ss:change level}), therefore it suffices to cover $\Sht^{(I_{1},\cdots, I_{r})}_{G,\bK_{D}}$ by open substacks of finite type over $k$. We have the forgetful maps $\Sht^{(I_{1},\cdots, I_{r}),\le\un\l}_{G,\bK_{D}}\xr{p_{0}}\Bun_{G,D}\to \Bun_{G}$. For $\mu\in \HN$, taking ${}^{\le \mu}\Sht^{(I_{1},\cdots, I_{r}),\le\un\l}_{G,\bK_{D}}$ to be the preimage of ${}^{\le\mu}\Bun_{G}$, then ${}^{\le \mu}\Sht^{(I_{1},\cdots, I_{r}),\le\un\l}_{G,\bK_{D}}$ is open in $\Sht^{(I_{1},\cdots, I_{r}),\le\un\l}_{G,\bK_{D}}$ and is of finite type over $k$. As $\mu\in \HN$ varies, ${}^{\le \mu}\Sht^{(I_{1},\cdots, I_{r}),\le\un\l}_{G,\bK_{D}}$ cover $\Sht^{(I_{1},\cdots, I_{r}),\le\un\l}_{G,\bK_{D}}$.

\sss{} Varshavsky \cite[Prop.2.16(1)]{Var} actually shows a more precise result:  each finite type open substack  ${}^{\le\mu}\Sht^{(I_{1},\cdots, I_{r}), \le\un\l}_{G,\bK}$ is the quotient of a  quasi-projective scheme by a finite (constant) group.  The idea of the argument is: by passing to deeper level structures $\bK'_{x}$ (still at points $x\in \Sig$), normal in $\bK_{x}$, one can arrange ${}^{\le\mu}\Bun_{G,\bK'}$ to be a quasi-projective scheme. It is then clear from the diagram \eqref{defn Sht} that  ${}^{\le\mu}\Sht^{(I_{1},\cdots, I_{r}), \le\un\l}_{G,\bK'}$ is a quasi-projective scheme. By the discussion in \S\ref{ss:change level}, the map ${}^{\le\mu}\Sht^{(I_{1},\cdots, I_{r}), \le\un\l}_{G,\bK'}\to {}^{\le\mu}\Sht^{(I_{1},\cdots, I_{r}), \le\un\l}_{G,\bK}$ is a torsor for the finite group $\prod_{s\in \Sig}K_{x}/K'_{x}$, where $K_{x}=\bK_{x}(k_{x})$, and $K'_{x}=\bK'_{x}(k_{x})$.

\subsection{Local model}
One would like to know finer geometric information about the moduli of Shtukas such as its smoothness and its dimension. These can be understood using local models given by  affine Schubert varieties in the affine Grassmannian.

Let $I=I_{1}\coprod\cdots\coprod I_{r}$ be a decomposition into {\em singletons}. Recall we have the relative position map \eqref{ev disj}:
\begin{equation}\label{ev Hk singleton}
\ev^{(I_{1},\cdots, I_{r})}_{\Hk}: \Hk^{(I_{1},\cdots, I_{r})}_{G,\bK}\to (\QGr/\Aut(D))^{I}.
\end{equation}
Note this map is defined over the whole $U^{I}$ since each $I_{j}$ is a singleton. Consider the composition
\begin{equation}\label{ev Sht singleton}
\ev^{(I_{1},\cdots, I_{r})}_{\Sht}: \Sht^{(I_{1},\cdots, I_{r})}_{G,\bK}\xr{} \Hk^{(I_{1},\cdots, I_{r})}_{G,\bK}\xr{\ev^{(I_{1},\cdots, I_{r})}_{\Hk}} (\QGr/\Aut(D))^{I}.
\end{equation}

\begin{prop}\label{p:fm sm} Let $I=I_{1}\coprod\cdots\coprod I_{r}$ be a decomposition into  singletons.  Then the maps \eqref{ev Hk singleton} and \eqref{ev Sht singleton} are formally smooth.
\end{prop}
\begin{proof}[Sketch of proof] 

We first show that \eqref{ev Sht singleton} is formally smooth.

Recall $D=\Spec k\tl{t}$. Let $\wh U\to U$ be the canonical $\Aut(D)$-torsor  $(x,t)$ where $t$ is a local parameter at $x\in U$. Then we have a map
\begin{equation*}
\wh\ev: \wh\Sht^{(I_{1},\cdots, I_{r})}_{G,\bK}:=\wh U^{I}\times_{U^{I}}\Sht^{(I_{1},\cdots, I_{r})}_{G,\bK}\to  \wh U^{I}\times \QGr^{I}
\end{equation*}
such that $\ev^{(I_{1},\cdots, I_{r})}$ is obtained by projecting to $\QGr^{I}$ and quotienting by $\Aut(D)^{I}$ on both sides. Therefore it suffices to show $\wh\ev$ is formally smooth.

For notational simplicity we only deal with the case without level structure, and the general case is similar.  We may identify $I$ with $\{1,\cdots, r\}$ so that $I_{j}=\{j\}$. Let $R'$ be a $k$-algebra, $J\subset R'$ a square-zero ideal and $R=R'/J$. Suppose we are given $R'$-points $(x'_{j}, t'_{j})\in \wh X(R')$ for $1\le j\le r$,  $R'$-points $\z'_{j}=(\cF'_{j}\dxr{} \cG'_{j})\in \QGr(R')$ for $1\le j\le r$ (where $(\cF'_{j}, \cG'_{j})$ is a pair of $G$-torsors over $D_{R'}$ with an isomorphism over $D^{\times}_{R'}$), and an $R$-point $\xi=((x_{j},t_{j})_{1\le j\le r}, (\cE_{j})_{0\le j\le r}, (\a_{j})_{1\le j\le r}, \io)$ of $\wh \Sht^{(I_{1},\cdots, I_{r})}_{G}$ such that $\wh\ev(\xi)=((x'_{j}, t'_{j},\z'_{j})|_{\Spec R})_{1\le j\le r}$. In particular, $(x_{j},t_{j})=(x'_{j}, t'_{j})|_{\Spec R}$. We want to lift $\xi$ to an $R'$-point $\xi'$ of $\wh \Sht^{(I_{1},\cdots, I_{r})}_{G}$ such that $\wh\ev(\xi')=(x'_{j}, t'_{j}, \z'_{j})_{1\le j\le r}$. This means we need to lift $\cE_{j}$ to a $G$-torsor $\cE'_{j}$ over $X_{R'}$ for $0\le j\le r$, lift the isomorphisms $\a_{j}$ to $\a'_{j}: \cE'_{j-1}|_{X_{R'}-\G(x'_{j})}\isom \cE'_{j}|_{X_{R'}-\G(x'_{j})}$ for $1\le j\le r$, and lift the isomorphism $\io: \cE_{r}\cong(\id_{X}\times \Fr_{R})^{*}\cE_{0}$ to an isomorphism $\io': \cE'_{r}\cong(\id_{X}\times \Fr_{R'})^{*}\cE'_{0}$.

First, $\cE'_{r}$ is already determined by $\cE_{0}$ as follows. Since $J^{2}=0$ is a square-zero thickening, $\Fr_{R'}$ factors as $R'\surj R\xr{\ov\Fr_{R'}} R'$. Let $\ph: \Spec R'\to \Spec R$ be the map given by $\ov\Fr_{R'}$. If $\cE'_{0}$ is any extension of $\cE_{0}$ to $X_{R'}$, we have $(\id_{X}\times \Fr_{R'})^{*}\cE'_{0}=(\id_{X}\times \ph)^{*}\cE_{0}$. Therefore, in order to have the desired isomorphism $\io': \cE'_{r}\cong {}^{\t}\cE'_{0}$, we must define $\cE'_{r}$ to be $(\id_{X}\times \ph)^{*}\cE_{0}$.

Next we extend the modifications $\a_{j}: \cE_{j-1}\dxr\cE_{j}$ along $\G(x_{j})$ to a modification $\a'_{j}: \cE'_{j-1}\dxr\cE'_{j}$ along $\G(x'_{j})$. We do this inductively, starting from the case $j=r$. In each step we are given $\cE'_{j}$ over $X_{R'}$ and try to construct $\cE'_{j-1}$ and $\a'_{j}$. Use $t'_{j}$ to identify $D_{x'_{j}}$ with $D_{R'}$, hence also $D_{x_{j}}$ with $D_{R}$. The isomorphism $\wh\ev(\xi)=(x_{j}, t_{j},\z'_{j}|_{\Spec R})_{1\le j\le r}$ gives isomorphisms between diagrams
\begin{equation*}
\xymatrix{\cE_{j-1}|_{D_{x_{j}}}\ar[d]^{\e_{j-1}}_{\wr}\ar@{-->}[r]^{\a_{j}} & \cE_{j}|_{D_{x_{j}}}\ar[d]^{\e_{j}}_{\wr}\\
\cF'_{j}|_{D_{R}}\ar@{-->}[r]^{\z_{j}} & \cG'_{j}|_{D_{R}}
}
\end{equation*}
We first lift the isomorphism $\e_{j}$ to $\e'_{j}: \cE'_{j}|_{D_{x'_{j}}}\cong \cG'_{j}$ (which is possible because $L^{+}G$ is formally smooth). Then we define $\cE'_{j-1}$ to be the gluing of $\cE'_{j}|_{X_{R'}-\G(x'_{j})}$ and $\cF'_{j}$ along the isomorphism $\cE'_{j}|_{D^{\times}_{x'_{j}}}\xr{\e'_{j}} \cG'_{j}|_{D^{\times}_{R'}}\xr{\z_{j}^{-1}}\cF'_{j}|_{D^{\times}_{R'}}$. Then we have a canonical isomorphism $\cE'_{j-1}|_{X_{R}}\cong \cE_{j-1}$ because the latter is glued from $\cE_{j}$ and $\cF'_{j}|_{D_{R}}$ along $\cE_{j}|_{D^{\times}_{x_{j}}}\xr{\e_{j}}\cG'_{j}|_{D^{\times}_{R}}\xr{\z_{j}^{-1}}\cF'_{j}|_{D^{\times}_{R}}$. This completes the inductive construction of $\cE'_{j}$. Therefore \eqref{ev Sht singleton} is formally smooth.

To show that \eqref{ev Hk singleton} is formally smooth, the argument is similar: simply start with any lifting $\cE'_{r}$ of $\cE_{r}$ to $X_{R'}$, which exists because $\Bun_{G}$ is smooth.
\end{proof}

Passing to bounded versions, the formal smoothness proved in Prop. \ref{p:fm sm} implies the smoothness of the bounded version of the relative position map. More strongly, affine Grassmannian serve as \'etale local model for the moduli stack in the following sense.


\begin{prop}[{\cite[Theorem 2.20]{Var}},{\cite[Theorem 2.11]{VLaff}}]\label{p:local model} Let $I=I_{1}\coprod\cdots\coprod I_{r}$ be a decomposition of the finite set $I$, and $\un\l\in (\xcoch(T)^{+})^{I}$. 

\begin{enumerate}
\item The stack $\Sht^{(I_{1},\cdots, I_{r}), \le \un\l}_{G,\bK}$ is \'etale locally isomorphic to $\Gr^{(I_{1},\cdots, I_{r}), \le\un\l}_{G}$ over $U^{I}$, which in turn is \'etale locally isomorphic to the product $\prod_{j=1}^{r}\Gr^{I_{j},\le\un\l_{j}}_{G}$ (where $\un\l_{j}=\{\l_{i}\}_{i\in I_{j}}$). 
\item In particular, if $\un\l$ is $G$-admissible, then 
\begin{equation*}
\dim \Sht^{(I_{1},\cdots, I_{r}),\le\un\l}_{G,\bK}=\sum_{i\in I}(\j{2\r, \l_{i}}+1).
\end{equation*}
\item Let $I_{j}=I_{j,1}\coprod\cdots\coprod I_{j,s_{j}}$ be a decomposition for each $j$, and let  $I=I'_{1}\coprod\cdots I'_{r'}$ be the resulting refinement of $I=I_{1}\coprod\cdots\coprod I_{r}$. Then the morphism $\pi^{(I'_{1},\cdots, I'_{r'}),\le\un\l}_{(I_{1},\cdots, I_{r})}: \Sht^{(I'_{1},\cdots, I'_{r'}), \le \un\l}_{G,\bK}\to \Sht^{(I_{1},\cdots, I_{r}), \le \un\l}_{G,\bK}$ is \'etale locally on the target isomorphic to the product of the natural maps $\prod_{j=1}^{r}\Gr^{(I_{j,1},\cdots, I_{j,s_{j}}),\le\un\l_{j}}_{G}\to \prod_{j=1}^{r}\Gr^{I_{j},\le\un\l_{j}}_{G}$.
\end{enumerate}
\end{prop} 
The proof eventually boils down to a tangential calculation, and uses that the differential of the Frobenius map is zero. 
 
If $\l_{i}$ are minuscule, then $\Gr^{\le\l_{i}}_{G}$ are smooth, and hence $\Gr^{(I_{1},\cdots, I_{r}), \le\un\l}_{G}$ is smooth over $U^{I}$. In this case, Prop. \ref{p:local model} implies:
\begin{cor}\label{c:Sht sm} Let $I=I_{1}\coprod\cdots\coprod I_{r}$ be a decomposition into  singletons, and $\l_{i}\in \xcoch(T)^{+}$ be a minuscule dominant coweight for all $i\in I$. Then $\Sht^{(I_{1},\cdots, I_{r}),\le\un\l}_{G,\bK}$ is smooth of pure relative dimension $\sum_{i\in I}\j{2\r, \l_{i}}$ over $U^{I}$.
\end{cor}

\subsection{Factorization structure}\label{ss:geom fact}
For a collection of geometric objects $A_{I}$ (stacks or sheaves) over $U^{I}$, one  for each finite set $I$, a factorization structure on $\{A_{I}\}$ usually has two aspects: the first is the behavior of $A_{I}$ when restricted to $U^{I}_{\disj}$ (for example, if $A_{I}$ is a stack over $U^{I}$, we may require $A_{I}|_{U^{I}_{\disj}}$ to be equipped with an isomorphism with $(\prod_{i\in I}A_{\{i\}})|_{U^{I}_{\disj}}$); the second is the behavior of $A_{I}$ over the partial diagonals (for example, we may require $A_{I}|_{\D(U)}\cong A_{\{1\}}$).  A typical example of such a factorization structure is carried by the Beilinson-Drinfeld Grassmannian $\Gr^{I}$ over $X^{I}$. 

In the case of Shtukas, the first aspect mentioned above holds only in a weak sense, which is essentially the content of Proposition \ref{p:fm sm} and Proposition \ref{p:local model}. We spell out the second aspect.

For any stack $\frX$, and any map of finite sets $\th: I\to J$, we have the induced map
\begin{equation*}
\D_{\th, \frX}: \frX^{J}\to \frX^{I}
\end{equation*}
sending  $(x_{j})_{j\in J}$ to $(x_{\th(i)})_{i\in I}$. When $\th$ is a surjection, $\D_{\th,\frX}$ is a partial diagonal in $\frX^{I}$.

The following property of the Hecke stack is clear from the definitions. If $\th: I\surj J$ is a surjection of finite sets, then there is a canonical isomorphism
\begin{equation}\label{Hk diagonal}
\Hk^{I}_{G,\bK}|_{\D_{\th,U}(U^{J})}\cong \Hk^{J}_{G,\bK}.
\end{equation}
Moreover, this isomorphism is compatible with a composition of surjections $I\surj J\surj K$.

In particular, taking $J$ to be a singleton, we conclude that
\begin{equation*}
\Hk^{I}_{G,\bK}|_{\D(U)}\cong \Hk^{\{1\}}_{G,\bK}.
\end{equation*}

Similar properties hold for $\Sht^{I}_{G,\bK}$. If $\th: I\surj J$ is a surjection of finite sets, then there is a canonical isomorphism
\begin{equation*}
\Sht^{I}_{G,\bK}|_{\D_{\th,U}(U^{J})}\cong \Sht^{J}_{G,\bK}.
\end{equation*}
Moreover, this isomorphism is compatible with a composition of surjections $I\surj J\surj K$.

In particular, taking $J$ to be a singleton, we conclude that
\begin{equation*}
\Sht^{I}_{G,\bK}|_{\D(U)}\cong \Sht^{\{1\}}_{G,\bK}.
\end{equation*}

\subsection{Change of level structure}\label{ss:change level}

\sss{Level raising} For each $x\in \Sig$, let $\bK^{\#}_{x}\lhd \bK_{x}$ be a normal subgroup with finite dimensional quotient $\bL_{x}=\bK_{x}/\bK^{\#}_{x}$. Let $\bK^{\#}=(\bK^{\#}_{x})_{x\in \Sig}$. Then the natural map
\begin{equation*}
\Sht^{(I_{1},\cdots, I_{r})}_{G,\bK^{\#}} \to \Sht^{(I_{1},\cdots, I_{r})}_{G,\bK} 
\end{equation*}
is a torsor for the finite group $L:=\prod_{x\in \Sig}\bL_{x}(k_{x})$. One can take the projective limit over such normal subgroups  $\bK^{\#}$ to get a stack over $U^{I}$
\begin{equation*}
\Sht^{(I_{1},\cdots, I_{r})}_{G,\infty\Sig} :=\varprojlim_{\bK^{\#}}\Sht^{(I_{1},\cdots, I_{r})}_{G,\bK^{\#}}.
\end{equation*}
The projection
\begin{equation*}
\Sht^{(I_{1},\cdots, I_{r})}_{G,\infty\Sig}\to \Sht^{(I_{1},\cdots, I_{r})}_{G,\bK}
\end{equation*}
is a pro-\'etale map that is a torsor for $K=\prod_{x\in\Sig}K_{x}$ where $K_{x}=\bK_{x}(k_{x})$. In particular, we have:

\begin{lemma}\label{l:change level Sig}
Let $\bK'_{x}\subset \bK_{x}$ be level groups for each $x\in \Sig$, and let  $\bK'=(\bK'_{x})_{x\in \Sig}$. Then the natural map
\begin{equation*}
\Sht^{(I_{1},\cdots, I_{r})}_{G,\bK'} \to \Sht^{(I_{1},\cdots, I_{r})}_{G,\bK} 
\end{equation*}
is finite \'etale whose geometric fibers are in bijection with cosets $\prod_{x\in \Sig}K_{x}/K'_{x}$ (where $K'_{x}=\bK'_{x}(k_{x})$). 
\end{lemma}

\sss{Full level at $x$}\label{sss:full x}
If $x\notin \Sig$, we can define a stack over $(U-\{x\})^{I}$
\begin{equation*}
\Sht^{(I_{1},\cdots, I_{r})}_{G,\bK,\infty x}=\varprojlim_{\bK^{\da}_{x}}\Sht^{(I_{1},\cdots, I_{r})}_{G,\bK^{\da}}
\end{equation*}
where $\bK^{\da}$ denotes the level structure along $\{x\}\cup \Sig$, with $\bK$ on $\Sig$ and $\bK^{\da}_{x}$ at $x$. This is a $G(\cO_{x})$-torsor over $\Sht^{(I_{1},\cdots, I_{r})}_{G, \bK}|_{(U-\{x\})^{I}}$.

We claim that the $G(\cO_{x})$ action on $\Sht^{(I_{1},\cdots, I_{r})}_{G, \bK,\infty x}$ extends canonically to an action of $G(F_{x})$. Indeed, suppose $\xi=((x_{i})_{i\in I}, (\cE_{j})_{0\le j\le r}, (\a_{j})_{1\le j\le r}, \io, (\ka_{j})_{0\le j\le r})$ is an $R$-point of $\Sht^{(I_{1},\cdots, I_{r})}_{G,\bK,\infty x}$, where $\ka_{j}: G\times D_{x, R}\isom \cE_{j}|_{D_{x,R}}$ is the trivialization corresponding to the full level structure at $x$. Here $D_{x,R}=\Spec(\cO_{x}\wh\ot R)$ and $D_{x,R}^{\times}=\Spec(F_{x}\wh\ot R)$. Let $g\in G(F_{x})$. We define $\xi\cdot g$ to be the tuple $((x_{i})_{i\in I}, (\cE^{g}_{j})_{0\le j\le r}, (\a'_{j})_{1\le j\le r}, \io', (\ka'_{j})_{0\le j\le r})$. Here $\cE^{g}_{j}$ is the $G$-torsor obtained by gluing $\cE_{j}|_{(X-\{x\})_{R}}$ with the trivial $G$-torsor over $D_{x,R}$ using the isomorphism
\begin{equation*}
G\times D^{\times}_{x, R}\xr{g} G\times D^{\times}_{x, R} \xr{\ka_{j}} \cE_{j}|_{D^{\times}_{x,R}}.
\end{equation*}
The first map above is the left multiplication by $g$. The maps $\a'_{j}$ and $\io'$ are induced by $\a_{j}$ and $\io$ because $\G(x_{j})$ is disjoint from $\{x\}_{R}$ by definition. Finally, the trivialization $\ka'_{j}$ of $\cE'_{j}|_{D_{x,R}}$ is the tautological one from the construction of $\cE'_{j}$ by gluing.

\subsection{Central action}\label{sss:central action}

Let $Z$ be the center of $G$. For each $x\in \Sig$ let $\bK_{Z,x}=\bK_{x}\cap L^{+}_{x}Z$ be a level group for $Z$ at $x$. Let $\bK_{Z}$ be the collection of $\bK_{Z,x}$ for $x\in \Sig$. Since $Z$ is abelian, the stack $\Bun_{Z,\bK_{Z}}$ is a group stack. In particular, $\Bun_{Z,\bK_{Z}}(k)$ is a Picard groupoid. 

There is a canonical action of $\Bun_{Z,\bK_{Z}}$ on $\Bun_{G,\bK}$. To see this, we first consider the case $\bK_{x}$ is a principal congruence subgroup $L^{(n_{x})}_{x}G$ for some $n_{x}\in\NN$, for each $x\in \Sig$. In this case, 
$\Bun_{G,\bK}=\Bun_{G,D}$ where $D=\sum n_{x}x$, and $\Bun_{Z,\bK_{Z}}=\Bun_{Z,D}$. Consider an $S$-point $\cE\in\Bun_{G,D}(S)$, corresponding to a right $G$-torsor $\cE$ on $X\times S$ with a trivialization $\io_{D}$ of $\cE|_{D\times S}$, and an  $S$-point $\cQ\in \Bun_{Z,D}(S)$, corresponding to a right $Z$-torsor $\cQ$ on $X\times S$ with a trivialization $\ka_{D}$ of $\cQ|_{D\times S}$. Since $Z$ is central in $G$, its natural action on $\cE$ commutes with the action of $G$, hence $Z$ acts on $\cE$ by $G$-torsor automorphisms. We form the twisted fiber product ${}_{\cQ}\cE:=\cQ\times^{Z}_{X\times S}\cE$, the quotient of $\cQ\times_{X\times S}\cE$ by the diagonal action of $Z$. The right action of $G$ on $\cE$ equips ${}_{\cQ}\cE$ with a $G$-torsor structure. Using the trivializations $\io_{D}$ and $\ka_{D}$ of $\cE|_{D\times S}$ and $\cQ|_{D\times S}$ respectively, we get a trivialization $\io'_{D}$ of ${}_{\cQ}\cE|_{D\times S}$. We define the result of the action $(\cQ,\ka_{D})\cdot (\cE,\io_{D})$ to be $({}_{\cQ}\cE, \io'_{D})$.

For more general level structures, we still get the desired action by choosing a sufficient small principal congruence subgroup $L^{(n_{x})}_{x}G$ contained in $\bK_{x}$ for each $x$, and identifying $\Bun_{G,\bK}$ (resp. $\Bun_{Z,\bK_{Z}}$) as a quotient of $\Bun_{G,D}$ (resp. $\Bun_{Z,D}$) by an algebraic subgroup of $G_{D}$ (resp. $Z_{D}$), as in \S\ref{sss:level}. We omit details.

Now consider the Shtuka version of the above construction.   We claim that there is a canonical action of $\Bun_{Z,\bK_{Z}}(k)$ on $\Sht^{(I_{1},\cdots, I_{r})}_{G,\bK}$. Indeed, for any $S$-point $\xi=((x_{i})_{i\in I}, (\cE_{j})_{0\le j\le r}, (\a_{j})_{1\le j\le r}, \io)$ of $\Sht^{(I_{1},\cdots, I_{r})}_{G,\bK}$, and $\cQ\in \Bun_{Z,\bK_{Z}}(k)$, we apply the above construction to obtain ${}_{\cQ}\cE_{j}\in \Bun_{G,\bK}(S)$. The original $\a_{j}: \cE_{j-1}|_{X\times S-\cup_{i\in I_{j}}\G(x_{i})}\isom  \cE_{j}|_{X\times S-\cup_{i\in I_{j}}\G(x_{i})}$ induces an isomorphism $\a'_{j}: {}_{\cQ}\cE_{j-1}|_{X\times S-\cup_{i\in I_{j}}\G(x_{i})}\isom  {}_{\cQ}\cE_{j}|_{X\times S-\cup_{i\in I_{j}}\G(x_{i})}$, and similarly $\io:\cE_{r}\isom{}^{\t}\cE_{0}$ induces an isomorphism $\io': {}_{\cQ}\cE_{r}\cong {}_{{}^{\t}\cQ}({}^{\t}\cE_{0})={}^{\t}({}_{\cQ}\cE_{0})$, because ${}^{\t}\cQ=\cQ$. The action of $\cQ$ on $\Sht^{(I_{1},\cdots, I_{r})}_{G,\bK}$ then takes $\xi$ to $\xi':=((x_{i})_{i\in I}, ({}_{\cQ}\cE_{j})_{0\le j\le r}, (\a'_{j})_{1\le j\le r}, \io')$.

It is easy to see that the $\Bun_{Z,\bK_{Z}}(k)$-action on $\Sht^{(I_{1},\cdots, I_{r})}_{G,\bK}$ preserves the bounded versions $\Sht^{(I_{1},\cdots, I_{r}),\le\un\l}_{G,\bK}$ for any collection of dominant coweights $\un\l\in (\xcoch(T)^{+})^{I}$.

\begin{exam} Consider the case $G=\GL_{n}$ without level structures. We have $\Bun_{Z}=\Pic_{X}$. The action of a line bundle $\cL\in \Pic_{X}(k)$ on $\Sht^{(I_{1},\cdots, I_{r})}_{\GL_{n}}$ is by tensoring each vector bundle $\cV_{j}$ with $\cL$.

On the other hand, consider the case $G=\SL_{n}$ without level structures. Then $\Bun_{Z}\cong \Pic_{X}[n]$ is the $n$-torsion substack of the Picard stack of $X$. For $\cL\in \Bun_{Z}(k)=\Pic_{X}[n](k)$, i.e., a line bundle $\cL$ equipped with an isomorphism $\e: \cO_{X}\cong \cL^{\ot n}$, the action of $\cL$ on $\Sht^{(I_{1},\cdots, I_{r})}_{\SL_{n}}$ is again given by tensoring each vector bundle $\cV_{j}$ with $\cL$. Note that the original trivialization of $\det(\cV_{j})$ together with $\e$  induces a trivialization of $\det(\cV_{j}\ot \cL)\cong \det(\cV_{j})\ot \cL^{\ot n}$.
\end{exam}

\subsection{Hecke correspondence}

\sss{Spherical Hecke correspondence}
For $x\notin \Sig$ let $K_{x}=G(\cO_{x})$. Let $[g_{x}]\in K_{x}\bs G(F_{x})/K_{x}$. We define
\begin{equation*}
 \Sht^{(I_{1},\cdots, I_{r})}_{G,\bK}[g_{x}]:= \Sht^{(I_{1},\cdots, I_{r})}_{G,\bK,\infty x}/(K_{x}\cap g_{x}K_{x}g_{x}^{-1}).
\end{equation*}
We view $\Sht^{(I_{1},\cdots, I_{r})}_{G,\bK}[g_{x}]$ as a correspondence
\begin{equation*}
\xymatrix{ & \Sht^{(I_{1},\cdots, I_{r})}_{G,\bK}[g_{x}]\ar[dl]_{\oll{c}}\ar[dr]^{\orr{c}}\\
\Sht^{(I_{1},\cdots, I_{r})}_{G,\bK}|_{(U-\{x\})^{I}} &&  \Sht^{(I_{1},\cdots, I_{r})}_{G,\bK}|_{(U-\{x\})^{I}}
}
\end{equation*}
where $\oll{c}$ is the natural projection corresponding to the inclusion  $K_{x}\cap g_{x}K_{x}g_{x}^{-1}$; the map $\orr{c}$ is the composition
\begin{equation*}
 \Sht^{(I_{1},\cdots, I_{r})}_{G,\bK,\infty x}/(K_{x}\cap g_{x}K_{x}g_{x}^{-1})\xr{\cdot g_{x}} \Sht^{(I_{1},\cdots, I_{r})}_{G,\bK,\infty x}/(g_{x}^{-1}K_{x}g_{x}\cap K_{x})\to \Sht^{(I_{1},\cdots, I_{r})}_{G,\bK}.
\end{equation*}
Here, for the first map, we are using the right $G(F_{x})$-action on $\Sht^{(I_{1},\cdots, I_{r})}_{G,\bK,\infty x}$ constructed in \S\ref{sss:full x}; the second map corresponds to the inclusion  $g^{-1}_{x}K_{x}g_{x}\cap K_{x}\subset K_{x}$.

\sss{Hecke correspondence at a ramified point} For  $x\in \Sig$ and $[g_{x}]\in  K_{x}\bs G(F_{x})/K_{x}$, we define a correspondence
\begin{equation}\label{Hk x Sig}
\xymatrix{ & \Sht^{(I_{1},\cdots, I_{r})}_{G,\bK}[g_{x}]\ar[dl]_{\oll{c}}\ar[dr]^{\orr{c}}\\
\Sht^{(I_{1},\cdots, I_{r})}_{G,\bK} &&  \Sht^{(I_{1},\cdots, I_{r})}_{G,\bK}
}
\end{equation}
by setting
\begin{equation*}
\Sht^{(I_{1},\cdots, I_{r})}_{G,\bK}[g_{x}]:=\Sht^{(I_{1},\cdots, I_{r})}_{G,\bK^{x}, \infty x}/(K_{x}\cap g_{x}K_{x}g_{x}^{-1}).
\end{equation*}
Here $\bK^{x}$ is the collection of level groups $(\bK_{x})_{x\in \Sig-\{x\}}$. The map $\oll{c}$ is the natural projection, and the map $\orr{c}$ is the right action of $g_{x}$ followed by the natural projection, as in the case of the spherical Hecke modification.

For a $k$-scheme $S$, $\Sht^{(I_{1},\cdots, I_{r})}_{G,\bK}[g_{x}](S)$ classifies $(x_{i})_{i\in I}$ (where $x_{i}\in U(S)$) together with a diagram
\begin{equation*}
\xymatrix{ \cE_{0}\ar@{-->}[d]^{f_{0}} \ar@{-->}[r]^{\a_{1}} & \cE_{1}\ar@{-->}[d]^{f_{1}} \ar@{-->}[r]^{\a_{2}} & \cdots  \ar@{-->}[r]^{\a_{r}}  & \cE_{r}\ar@{-->}[d]^{f_{r}}\ar[r]^{\io}_{\sim} & {}^{\t}\cE_{0}\ar@{-->}[d]^{{}^{\t}f_{0}}\\
 \cE'_{0}\ar@{-->}[r]^{\a'_{1}}  & \cE'_{1}\ar@{-->}[r]^{\a'_{1}}  & \cdots \ar@{-->}[r]^{\a'_{r}}  & \cE'_{r}\ar[r]^{\io'}_{\sim}  & {}^{\t}\cE'_{0}}
\end{equation*}
Here the top and bottom rows are both $S$-points of $\Sht^{(I_{1},\cdots, I_{r})}_{G,\bK}$ with legs $(x_{i})_{i\in I}$, and the vertical maps are isomorphisms of $G$-torsors on $U\times S$\begin{equation*}
f_{i}: \cE_{i}|_{U\times S}\isom \cE'_{i}|_{U\times S}
\end{equation*}
which extends to an isomorphism of $\bK_{x'}$-level structures for $x'\in \Sigma-\{x\}$, and has  relative position $[g_{x}]\in K_{x}\bs G(F_{x})/K_{x}$ at $x$.

\sss{Bounded version} Given a bound $\un\l\in (\xcoch(T)^{+})^{I}$, we have a bounded version of the Hecke correspondences with bound $\un\l$:
\begin{equation*}
\Sht^{(I_{1},\cdots, I_{r}), \le \un\l}_{G,\bK}[g_{x}]
\end{equation*}
as a self correspondence of $\Sht^{(I_{1},\cdots, I_{r}), \le \un\l}_{G,\bK}|_{(U-\{x\})^{I}}$ (if $x\notin \Sig$) or $\Sht^{(I_{1},\cdots, I_{r}), \le \un\l}_{G,\bK}$ (if $x\in \Sig$).

\subsection{Partial Frobenius}
There is a canonical morphism
\begin{equation}\label{pl Fr}
\Fr^{\pl}: \Sht^{(I_{1},\cdots, I_{r})}_{G,\bK}\to \Sht^{(I_{2},\cdots, I_{r}, I_{1})}_{G,\bK}
\end{equation}
defined on the level of $S$-points as follows: it sends an $S$-point 
\begin{equation*}
\xymatrix{\xi=((x_{i})_{i\in I}, \quad   \cE_{0}\ar@{-->}[r]^-{\a_{1}} & \cdots\ar@{-->}[r]^-{\a_{r-1}} & \cE_{r-1}\ar@{-->}[r]^-{\a_{r}} & \cE_{r}\ar[r]^-{\io}_{\sim} & {}^{\t}\cE_{0})}
\end{equation*}
to 
\begin{equation*}
\xymatrix{\Fr^{\pl}(\xi)=( (x'_{i})_{i\in I}, \quad  \cE_{1}\ar@{-->}[r]^-{\a_{2}} & \cdots\ar@{-->}[r]^-{\a_{r}} & \cE_{r}\ar@{-->}[r]^-{{}^{\t}\a_{1}\c\io} & {}^{\t}\cE_{1}\ar@{=}[r]^{\io'=\id} & {}^{\t}\cE_{1})}.
\end{equation*}
Here
\begin{equation*}
x'_{i}=\begin{cases} x_{i}\c\Fr_{S}, & i\in I_{1}\\
x_{i}, & i\notin I_{1}.\end{cases}
\end{equation*}
It is clear that
\begin{equation*}
\underbrace{\Fr^{\pl}\c \Fr^{\pl}\c\cdots \c \Fr^{\pl}}_{r \textup{ times}}= \Fr\in \End(\Sht^{(I_{1},\cdots, I_{r})}_{G,\bK})
\end{equation*}

We have a commutative diagram
\begin{equation}\label{pl Fr UI}
\xymatrix{\Sht^{(I_{1},\cdots, I_{r})}_{G,\bK}\ar[r]^{\Fr^{\pl}} \ar[d]^{\pi^{(I_{1},\cdots, I_{r})}} & \Sht^{(I_{2},\cdots, I_{1})}_{G,\bK}\ar[d]^{\pi^{(I_{2},\cdots, I_{1})}}\\
U^{I}\ar[r]^{\Fr_{I_{1}}} & U^{I}
}
\end{equation}
Here $\Fr_{I_{1}}:U^{I}\to U^{I}$ is the Frobenius on the $I_{1}$-factors and identity on the other factors. This diagram is topologically Cartesian because both $\Fr_{I_{1}}$ on $U^{I}$ and $\Fr^{\pl}$ are universal homeomorphisms.

\section{Structures on cohomology}\label{s:coho}

In this section, we summarize the basic symmetries on the cohomology of the moduli stacks of $G$-Shtukas (relative to the space of legs): factorization, Hecke algebra action and partial Frobenius action. These symmetries will be used in the further study of the cohomology of the moduli stacks of $G$-Shtukas in the lectures of C.Xue \cite{Xue}.

\subsection{Cohomology of the moduli of Shtukas}

\sss{Satake sheaves}
Recall the geometric Satake equivalence of Lusztig, Beilinson--Drinfeld, Ginzburg, Mirkovic--Vilonen (see \cite{MV}). We have an equivalence of tensor categories over $\Qlbar$:
\begin{equation*}
\Sat: \Rep(\dG,\Qlbar)\isom \Perv_{\Aut(D)}(\QGr_{G})
\end{equation*}
where the monoidal structure on the right side is given by convolution, and the commutativity constraint given by the fusion construction of Beilinson and Drinfeld. Here an $\Aut(D)$-equivariant perverse sheaf on $\QGr_{G}$ means a $L^{+}G\rtimes\Aut(D)$-equivariant perverse sheaf on $\Gr_{G}$. For any finite set $I$, we define a functor
\begin{equation*}
\Sat_{I}: \Rep(\dG^{I},\Qlbar)\to\Perv((\QGr_{G}/\Aut(D))^{I}).
\end{equation*}
For $V=\boxtimes_{i\in I}V_{i}$ where $V_{i}\in \Rep(\dG)$, define $\Sat_{I}=\boxtimes_{i\in I}\Sat(V_{i})$ and extend by additivity.

We shall define a complex of sheaves on $\Hk^{I}_{G,\bK}$. Let $r=|I|$, and choose a bijection $\s: I\isom \{1,2,\cdots, r\}$. Let $I_{j}=\{\s^{-1}(j)\}\subset I$, then we have a decomposition $I=I_{1}\coprod\cdots\coprod I_{r}$ into singletons. Recall the relative position map from \eqref{ev Hk singleton}
\begin{equation*}
\ev^{(I_{1},\cdots, I_{r})}_{\Hk}: \Hk^{(I_{1},\cdots, I_{r})}_{G,\bK}\to (\QGr_{G}/\Aut(D))^{I}.
\end{equation*}
We define the pullback
\begin{equation*}
\IC^{(I_{1},\cdots, I_{r})}_{\Hk}(V):=(\ev^{(I_{1},\cdots, I_{r})}_{\Hk})^{*}\Sat_{I}(V)\in D_{c}(\Hk^{(I_{1},\cdots, I_{r})}_{G,\bK}, \Qlbar).
\end{equation*}
If $V=\boxtimes_{i\in I}V_{i}$, and $V_{i}$ is irreducible with highest weight $\l_{i}\in\xcoch(T)^{+}$, then $\IC^{(I_{1},\cdots, I_{r})}_{\Hk}(V)$ is supported on the closed substack $\Hk^{(I_{1},\cdots, I_{r}), \le\un\l}_{G,\bK}$ with $\un\l=(\l_{i})_{i\in I}$.

We have the projection
\begin{equation*}
\pi^{(I_{1},\cdots, I_{r})}_{I}: \Hk^{(I_{1},\cdots, I_{r})}_{G,\bK}\to \Hk^{I}_{G,\bK}
\end{equation*}
Then define
\begin{equation*}
\IC^{I,\s}_{\Hk}(V):=\pi^{(I_{1},\cdots, I_{r})}_{I, !}\IC^{(I_{1},\cdots, I_{r})}_{\Hk}(V)\in D_{c}(\Hk^{I}_{G,\bK}, \Qlbar).
\end{equation*}
Again, if $V=\boxtimes_{i\in I}V_{i}$, and $V_{i}$ is irreducible with highest weight $\l_{i}\in\xcoch(T)^{+}$, then $\IC^{I}_{\Hk}(V)$ is supported on  $\Hk^{I, \le\un\l}_{G,\bK}$ with $\un\l=(\l_{i})_{i\in I}$.

\begin{lemma}\label{l:IC Hk}
For any $V\in \Rep(\dG^{I})$, the complex $\IC^{I,\s}_{\Hk}(V)$ is independent of the choice of the ordering $\s$ on $I$ up to a canonical isomorphism. We denote it by $\IC^{I}_{\Hk}(V)$.

Moreover, if $V=\boxtimes_{i\in I}V_{i}$ and  $V_{i}$ is irreducible with highest weight $\l_{i}$, then $\IC^{I,\s}_{\Hk}(V)$ is, up to a shift, isomorphic to the intersection complex of $\Hk^{I, \le\un\l}_{G,\bK}$, where $\un\l=(\l_{i})_{i\in I}$.
\end{lemma}
\begin{proof}[Sketch of proof] We assume $V=\boxtimes_{i\in I}V_{i}$ and $V_{i}$ is irreducible with highest weight $\l_{i}$. For a decomposition into singletons $I=I_{1}\coprod\cdots\coprod I_{r}$, the relative position map $\ev^{(I_{1},\cdots, I_{r})}_{\Hk}$ is formally smooth by Proposition \ref{p:fm sm}. Therefore $\IC^{(I_{1},\cdots, I_{r})}_{\Hk}(V)$ is, up to a shift, canonically isomorphic to the intersection complex of $\Hk^{(I_{1},\cdots, I_{r}), \le\un\l}_{G,\bK}$. The map $\pi^{(I_{1},\cdots, I_{r})}_{I}$ is stratified small in the sense of \cite[after Prop. 4.2]{MV}, which implies that $\IC^{I,\s}_{\Hk}(V)$ is, up to a shift, canonically isomorphic to the intersection complex of $\Hk^{I, \le\un\l}_{G,\bK}$. Hence it is independent of the choice of $\s$.
\end{proof}

More generally, if $I=I_{1}\coprod\cdots \coprod I_{r}$ is an arbitrary decomposition of $I$, then we can choose an ordering of each $I_{j}$ and refine the decomposition into a decomposition into singletons $I=I'_{1}\coprod\cdots \coprod I'_{s}$ (so there is a bijection $\s: I\isom \{1,2,\cdots, s\}$ such that $I'_{j}=\{\s^{-1}(j)\}$ for $1\le j\le s$, and $\s(I_{j})$ is a consecutive subset of $\{1,2,\cdots, s\}$). We have a canonical projection
\begin{equation*}
\pi^{(I'_{1},\cdots, I'_{s})}_{(I_{1},\cdots, I_{r})}: \Hk^{(I'_{1},\cdots, I'_{s})}_{G,\bK}\to \Hk^{(I_{1},\cdots, I_{r})}_{G,\bK}.
\end{equation*}
For $V\in \Rep(\dG^{I}, \Qlbar)$ define
\begin{equation*}
\IC^{(I_{1},\cdots, I_{r}),\s}_{\Hk}(V):=\pi^{(I'_{1},\cdots, I'_{s})}_{(I_{1},\cdots, I_{r}), !}\IC^{(I'_{1},\cdots, I'_{s})}_{\Hk}(V)\in D_{c}(\Hk^{(I_{1},\cdots, I_{r})}_{G,\bK}, \Qlbar).
\end{equation*}

The same argument of Lemma \ref{l:IC Hk} proves the following.

\begin{lemma}\label{l:IC Hk refine} Let $V\in \Rep(\dG^{I})$.
\begin{enumerate}
\item The complex $\IC^{(I_{1},\cdots, I_{r}),\s}_{\Hk}(V)$ on $\Hk^{(I_{1},\cdots, I_{r})}_{G,\bK}$ is independent of the choice of the ordering $\s$ on $I$ refining the decomposition $I=I_{1}\coprod\cdots \coprod I_{r}$, up to a canonical isomorphism. We denote this complex by $\IC^{(I_{1},\cdots, I_{r})}_{\Hk}(V)$.

\item Let $I=I'_{1}\coprod\cdots \coprod I'_{r'}$ be any refinement of the decomposition  $I=I_{1}\coprod\cdots \coprod I_{r}$. Then there is a canonical isomorphism
\begin{equation*}
\IC^{(I_{1},\cdots, I_{r})}_{\Hk}(V)\cong \pi^{(I'_{1},\cdots, I'_{r'})}_{(I_{1},\cdots, I_{r}), !}\IC^{(I'_{1},\cdots, I'_{r'})}_{\Hk}(V).
\end{equation*}
Moreover, these isomorphisms are compatible with further refinements.

\item If $V=\boxtimes_{i}V_{i}$ and each $V_{i}$ is irreducible with highest weight $\l_{i}$, then $\IC^{(I_{1},\cdots, I_{r})}_{\Hk}(V)$ is, up to a shift, isomorphic to the intersection complex of $\Hk^{(I_{1},\cdots, I_{r}),\le\un\l}_{G,\bK}$, where $\un\l=(\l_{i})_{i\in I}$.
\end{enumerate}
\end{lemma}


\sss{Cohomology complexes of $\Sht^{I}_{G}$}
Recall the canonical map
\begin{equation*}
\r^{(I_{1},\cdots, I_{r})}: \Sht^{(I_{1},\cdots, I_{r})}_{G,\bK}\to \Hk^{(I_{1},\cdots, I_{r})}_{G,\bK}.
\end{equation*}
For $V\in \Rep(\dG^{I}, \Qlbar)$ define
\begin{equation*}
\IC^{(I_{1},\cdots, I_{r})}_{\Sht}(V):=(\r^{(I_{1},\cdots, I_{r})})^{*}\IC^{(I_{1},\cdots, I_{r})}_{\Hk}(V)\in D_{c}(\Sht^{(I_{1},\cdots, I_{r})}_{G,\bK}).
\end{equation*}

Alternatively, when $I=I_{1}\coprod\cdots \coprod I_{r}$ is a decomposition into singletons, we can consider the relative position map \eqref{ev Sht singleton} for $\Sht^{(I_{1},\cdots, I_{r})}_{G,\bK}$.
Then
\begin{equation*}
\IC^{(I_{1},\cdots, I_{r})}_{\Sht}(V)\cong (\ev_{\Sht}^{(I_{1},\cdots, I_{r})})^{*}\Sat_{I}(V).
\end{equation*}

\begin{lemma}\label{l:IC Sht refine} Let $V\in \Rep(\dG^{I})$.
\begin{enumerate}
\item For any refinement $I=I'_{1}\coprod\cdots \coprod I'_{r'}$  of the decomposition  $I=I_{1}\coprod\cdots \coprod I_{r}$, there is a canonical isomorphism
\begin{equation*}
\IC^{(I_{1},\cdots, I_{r})}_{\Sht}(V)\cong \pi^{(I'_{1},\cdots, I'_{r'})}_{(I_{1},\cdots, I_{r}), !}\IC^{(I'_{1},\cdots, I'_{r'})}_{\Sht}(V)
\end{equation*}
compatible with further refinements.
\item If $V=\boxtimes_{i\in I}V_{i}$ where  $V_{i}$ is irreducible with highest weight $\l_{i}$, then $\IC^{(I_{1},\cdots, I_{r})}_{\Sht}(V)$ is, up to a shift by $\#I$, isomorphic to the intersection complex of $\Sht^{(I_{1},\cdots, I_{r}), \le\un\l}_{G,\bK}$, where $\un\l=(\l_{i})_{i\in I}$.
\end{enumerate}
\end{lemma}
\begin{proof}
(1) follows from the Lemma \ref{l:IC Hk refine} and proper base change for the Cartesian square
\begin{equation*}
\xymatrix{ \Sht^{(I'_{1},\cdots, I'_{r'})}_{G,\bK}\ar[rr]^-{\r^{(I'_{1},\cdots, I'_{r'})}} \ar[d]^{\pi^{(I'_{1},\cdots, I'_{r'})}_{(I_{1},\cdots, I_{r})}}&& \Hk^{(I'_{1},\cdots, I'_{r'})}_{G,\bK}\ar[d]^{\pi^{(I'_{1},\cdots, I'_{r'})}_{(I_{1},\cdots, I_{r})}}\\
\Sht^{(I_{1},\cdots, I_{r})}_{G,\bK}\ar[rr]^-{\r^{(I_{1},\cdots, I_{r})}} && \Hk^{(I_{1},\cdots, I_{r})}_{G,\bK}
}
\end{equation*}

(2) First assume $I_{j}$ are singletons, in which case the statement follows from Proposition \ref{p:fm sm}.

Now consider a general  decomposition $I=I_{1}\coprod\cdots\coprod I_{r}$. Decompose each $I_{j}$ into singletons, and combine them to obtain a decomposition $I=I'_{1}\coprod\cdots\coprod I'_{r'}$ where each $I'_{i}$ is a singleton, and $I_{j}$ is the union of consecutive parts. By Proposition \ref{p:local model}, the map $\pi: \Sht^{(I'_{1},\cdots, I'_{r'}),\le\un\l}_{G,\bK}\to\Sht^{(I_{1},\cdots, I_{r}),\le\un\l}_{G,\bK}$ is \'etale locally isomorphic to the map $\Gr^{(I'_{1},\cdots, I'_{r'}),\un\le\l}_{G}\to \Gr^{(I_{1},\cdots, I_{r}),\un\le\l}_{G}$, which is stratified small. Therefore by (1),
\begin{equation*}
\IC^{(I_{1},\cdots, I_{r})}_{\Sht}(V)\cong \pi_{!}\IC^{(I'_{1},\cdots, I'_{r'})}_{\Sht}(V)\cong\pi_{!}\IC(\Sht^{(I'_{1},\cdots, I'_{r'}),\le\un\l}_{G,\bK})[-\#I]\cong\IC(\Sht^{(I_{1},\cdots, I_{r}),\le\un\l}_{G,\bK})[-\#I].
\end{equation*}
\end{proof}

\begin{defn} Let $I$ be a finite set  with a decomposition $I=I_{1}\coprod\cdots \coprod I_{r}$. Recall the map $\pi^{I}: \Sht^{(I_{1},\cdots, I_{r})}_{G,\bK}\to U^{I}$.  Define the ind-constructible complex
\begin{equation*}
\cH^{(I_{1},\cdots, I_{r})}_{G,\bK}(V):=\pi^{I}_{!}\IC^{(I_{1},\cdots, I_{r})}_{\Sht}(V)\in D(U^{I}, \Qlbar).
\end{equation*}
When $(G,\bK)$ is fixed in the discussion, we simply write $\cH^{(I_{1},\cdots, I_{r})}_{G,\bK}(V)$ as $\cH^{(I_{1},\cdots, I_{r})}(V)$. 

In particular, we have the ind-constructible complex $\cH^{I}(V)\in D(U^{I}, \Qlbar)$. 
\end{defn}


\begin{remark}
To see that $\cH^{I}(V)$ is ind-constructible, we use the Harder-Narasimhan truncation discussed in \S\ref{sss:HN}. The general level structure can be reduced to the case of principal congruence level structures so we assume $\bK_{x}\subset L^{+}_{x}G$ for all $x\in \Sig$.   For $\mu\in \HN$, let ${}^{\le \mu}\pi^{I}: {}^{\le\mu}\Sht^{I}_{G,\bK}\to U^{I}$ be the restriction of $\pi^{I}$. Let ${}^{\le \mu}\IC_{\Sht}^{I}(V)$ be the restriction of $\IC^{I}_{\Sht}(V)$ to ${}^{\le\mu}\Sht^{I}_{G,\bK}$. Then ${}^{\le\mu}\cH^{I}(V):={}^{\le \mu}\pi^{I}_{!}({}^{\le \mu}\IC_{\Sht}^{I}(V))$ is constructible because the support of ${}^{\le \mu}\IC_{\Sht}^{I}(V)$ is of finite type over $U^{I}$. We then have
\begin{equation}\label{HV colim}
\cH^{I}(V)=\colim_{\mu\in \HN}{}^{\le\mu}\cH^{I}(V)
\end{equation}
as an ind-constructible complex on $U^{I}$. Here for $\mu\le \mu'$ the transition map ${}^{\le\mu}\cH^{I}(V)\to {}^{\le\mu'}\cH^{I}(V)$ is induced by the open embedding ${}^{\le\mu}\Sht^{I}_{G,\bK}\incl {}^{\le\mu'}\Sht^{I}_{G,\bK}$.  
\end{remark}

\begin{cor}[of Lemma \ref{l:IC Sht refine}]\label{c:same HV}
For any refinement $I=I'_{1}\coprod\cdots \coprod I'_{r'}$  of the decomposition  $I=I_{1}\coprod\cdots \coprod I_{r}$, there is a canonical isomorphism
\begin{equation*}
\cH^{(I_{1},\cdots, I_{r})}(V)\cong \cH^{(I'_{1},\cdots, I'_{r'})}(V).
\end{equation*}
compatible with further refinements. In particular, for any decomposition  $I=I_{1}\coprod\cdots \coprod I_{r}$, there is a canonical isomorphism
\begin{equation*}
\cH^{(I_{1},\cdots ,I_{r})}(V)\cong\cH^{I}(V)\in D(U^{I}, \Qlbar).
\end{equation*}
\end{cor}
\begin{proof} 
Apply $\pi^{(I_{1},\cdots, I_{r})}_{!}$ to the isomorphism in Lemma \ref{l:IC Sht refine}(1).
\end{proof}


\subsection{Hecke action}

Let $V\in \Rep(\dG^{I})$.

\sss{} For any $x\in \Sig$, the vector space $\bH_{x}:=\Qlbar[K_{x}\bs G(F_{x})/K_{x}]$ of compactly supported bi-$K_{x}$-invariant functions on $G(F_{x})$ is a $\Qlbar$-algebra under convolution, with unit $\one_{K_{x}}$. We have a ring homomorphism
\begin{equation*}
\bH_{x}\to \End(\cH^{I}(V))
\end{equation*}
defined as follows. 

By Corollary \ref{c:same HV}, it suffice to construct an action of $\cH_{x}$ on $\cH^{(I_{1},\cdots, I_{r})}(V)$, where $I_{j}$ are singletons. Also we may reduce to the case $V$ is irreducible hence $V=\boxtimes_{i\in I}V_{i}$ for irreducible representations $V_{i}$ of $\dG$ with highest weight $\l_{i}$. Let $\un\l=(\l_{i})_{i\in I}$. For $[g_{x}]\in K_{x}\bs G(F_{x})/K_{x}$, recall the Hecke correspondence \eqref{Hk x Sig} where both $\oll{c}$ and $\orr{c}$ are finite \'etale by Lemma \ref{l:change level Sig}. 

We can similarly construct the relative position map 
\begin{equation*}
\ev_{\Sht[g_{x}]}^{(I_{1},\cdots, I_{r})}: \Sht^{(I_{1},\cdots, I_{r})}_{G,\bK}[g_{x}]\to (\QGr_{G}/\Aut(D))^{I}
\end{equation*}
compatible with the maps $\oll{c}$ and $\orr{c}$ and the relative position map for $\Sht^{(I_{1},\cdots, I_{r})}_{G,\bK}$ as in \eqref{ev Sht singleton}. Thus we get a canonical isomorphism
\begin{equation}\label{corr IC}
\oll{c}^{*}\IC_{\Sht}^{I}(V)\isom \orr{c}^{*}\IC_{\Sht}^{I}(V)\cong\orr{c}^{!}\IC_{\Sht}^{I}(V).
\end{equation}
The action of $\one_{K_{x}g_{x}K_{x}}$ on $\cH^{I}(V)$ is defined to be the composition 
\begin{equation*}
\cH^{I}(V)=\pi^{I}_{!}\IC_{\Sht}^{I}(V)\to \pi^{I}_{!}\oll{c}_{!}\oll{c}^{*}\IC_{\Sht}^{I}(V)\xr{\eqref{corr IC}} \pi^{I}_{!}\orr{c}_{!}\orr{c}^{!}\IC_{\Sht}^{I}(V)\to \pi^{I}_{!}\IC_{\Sht}^{I}(V)=\cH^{I}(V)
\end{equation*}
Here we use the unit and counit of adjunctions $\id \to \oll{c}_{*}\oll{c}^{*}\cong \oll{c}_{!}\oll{c}^{*}$ (because $\oll{c}$ is proper), and $\orr{c}_{!}\orr{c}^{!}\to\id$.

\sss{Spherical Hecke action} For $x\in |U|=|X|-\Sig$, a similar construction gives a ring homomorphism
\begin{equation*}
\bH_{x}:=\Qlbar[G(\cO_{x})\bs G(F_{x})/G(\cO_{x})]\to \End(\cH^{I}(V)|_{(U-\{x\})^{I}}).
\end{equation*}

\begin{remark} The action of $\bH_{x}$ on $\cH^{I}(V)|_{(U-\{x\})^{I}}$ for $x\in |U|$ has been extended to the whole of $\cH^{I}(V)$ by V. Lafforgue's, by identifying it with a special case of an excursion operator.
\end{remark}

\sss{} Let $\Sig'\subset |U|$ be any finite set of places. Combining the Hecke actions defined above, we get a ring homomorphism
\begin{equation*}
\bH_{\Sig'}=\ot_{x\in \Sig'}\bH_{x}\to \End(\cH^{I}(V)|_{(U-\Sig')^{I}}).
\end{equation*}
Taking colimit as $\Sig'$ grows, we get an action of the global Hecke algebra on the restriction of $\cH^{I}(V)$ to the generic point $\y^{I}\to X^{I}$
\begin{equation*}
\bH_{G,\bK}:=\ot'_{x\in |X|}\bH_{x}\to \End(\cH^{I}(V)|_{\y^{I}}).
\end{equation*}

\subsection{Factorization structure}

\sss{The expectation}
Suppose $V=\boxtimes_{i\in I}V_{i}$ where $V_{i}\in \Rep(\dG)$. Then $\cH^{I}(V)$ is {\em expected} to satisfy the following factorization property.  Let $\pi$ be a cuspidal automorphic representation of $G(\AA)$. Let $\pi(\bK)\subset \pi$ be the invariant vectors under $\prod_{x\notin\Sig}G(\cO_{x})\times \prod_{x\in \Sig}K_{x}$, where $K_{x}=\bK_{x}(k_{x})$. Note that $\pi(\bK)$ is an $\bH_{G,\bK}$-module. By the work of V. Lafforgue \cite{VLaff}, one can attach a semisimple Langlands parameter $\r_{\pi}$ to $\pi$, which is a $\dG$-local system on $U$. Using excursion operators of V. Lafforgue,  it is possible to define the $\pi$-isotypic subsheaf of $H^{i}\cH^{I}(V)$ (the $i$th cohomology sheaf of $\cH^{I}(V)$, $i\in \ZZ$), which we denote by $H^{i}\cH^{I}(V)[\pi]$. As a first approximation, we can define $H^{i}\cH^{I}(V)[\pi]$ to be $\Hom_{\bH_{G,\bK}}(\pi(\bK), H^{i}\cH^{I}(V))$ (which may be larger than the actual $H^{i}\cH^{I}(V)[\pi]$; the subtlety being  that the $\bH_{G,\bK}$-module $\pi(\bK)$ may not determine the isomorphism class of $\r_{\pi}$).

We make a simplifying assumption that $\Aut(\r_{\pi})=Z\dG$. For example this is known to be the case when $G=\GL_{n}$ and $\pi$ is cuspidal (in which case $\r_{\pi}$ is an irreducible rank $n$ local system). Then it is expected that 
\begin{equation*}
H^{0}\cH^{I}(V)[\pi]\stackrel{?}{\simeq} M_{\pi}\ot \boxtimes_{i\in I}(V_{i,\r_{\pi}}),
\end{equation*}
where $V_{i,\r_{\pi}}$ is the local system on $U$ induced from $\r_{\pi}$ and the $\dG$-representation $V_{i}$, and $M_{\pi}$ is a finite-dimensional vector space depending on $\pi$. Also it is expected that in this case $H^{i}\cH^{I}(V)[\pi]=0$  for $i\ne0$. For more precise statements, which are analogs of Arthur's multiplicity formula for automorphic forms over number fields, we refer to the paper of Xinwen Zhu \cite[\S4.7]{Zhu}.

The above expectation in particular suggests that on the geometric generic stalk of $\cH^{I}(V)$, there are commuting actions of $I$ copies of the fundamental group $\pi_{1}(U,\ov \y)$.  This fact is proved by Cong Xue \cite{X}, generalizing the result of V. Lafforgue \cite{VLaff}.  Their proofs rely on partial Frobenius (to be discussed in \S\ref{sss:pF}) and Drinfeld's lemma.

\sss{Behavior along partial diagonals} Below we make precise the behavior of $\cH^{I}(V)$ when restricted to the partial diagonals.

A map of finite sets $\th: I\to J$ induces a group homomorphism $\D_{\dG, \th}: \dG^{J}\to \dG^{I}$, hence a restriction functor on representations
\begin{equation*}
\Res_{\th}: \Rep(\dG^{I})\to \Rep(\dG^{J}).
\end{equation*}
 
\begin{prop} 
For a surjection $\th: I\surj J$ of finite sets, we have the partial diagonal map $\D_{\th,U}: U^{J}\incl U^{I}$.  Then we have a canonical isomorphism
\begin{equation*}
\D_{\th,U}^{*}\cH^{I}(V)\isom \cH^{J}(\Res_{\th}V).
\end{equation*}
Moreover, this isomorphism is functorial in $V\in \Rep(\dG^{I})$, and compatible with compositions of surjections $I\surj J\surj K$.
\end{prop}
\begin{proof}[Proof sketch]
By proper base change, it suffices to check the corresponding statement for the Satake sheaves on $\Hk^{J}_{G,\bK}$:
\begin{equation*}
\IC^{I}_{\Hk}(V)|_{\Hk^{J}_{G,\bK}}\cong \IC^{J}_{\Hk}(\Res_{\th}(V)).
\end{equation*}
Here we use the isomorphism \eqref{Hk diagonal} to view $\Hk^{J}_{G,\bK}$ as a substack of $\Hk^{I}_{G,\bK}$ (restriction via $\D_{\th, U}: U^{J}\incl U^{I}$). Choose a bijection $\s: J\isom \{1,\cdots, r\}$ and let $I_{j}=\th^{-1}(\s'^{-1}(j))$ for $1\le j\le r'$.  Then we have a Cartesian diagram
\begin{equation*}
\xymatrix{ \Hk^{(\{1\},\cdots, \{r\})}_{G,\bK}\ar[r]^{\y}\ar[d] & \Hk^{(I_{1},\cdots, I_{r})}_{G,\bK}\ar[d]\\
\Hk^{J}_{G,\bK}\ar[r] & \Hk^{I}_{G,\bK}}
\end{equation*}
By proper base change and Lemma \ref{l:IC Hk refine}, it suffices to show that
\begin{equation}\label{IC fact b}
\y^{*}\IC^{(I_{1},\cdots, I_{r})}_{\Hk}(V)\cong \IC^{(\{1\},\cdots ,\{r\})}_{\Hk}(\Res_{\th}V).
\end{equation}

Below we treat the case $J$ itself is a singleton. The general case requires no more new idea. So we assume $J=\{1\}$ is a singleton. Choose a bijection $\s: I\isom \{1,\cdots, s\}$, and define 
\begin{equation*}
\Hk^{[s]}_{G,\bK}=\Hk^{(\{1\},\cdots,\{s\})}_{G,\bK}\times_{U^{s}}\D(U).
\end{equation*}
Then $\Hk^{[s]}_{G,\bK}$ classifies modifications of $G$-torsors with $\bK$-level structures $\cE_{0}\dxr\cE_{1}\dxr\cdots\dxr\cE_{s}$ all along the same point $x\in U$.  Let $\QGr^{[s]}_{G}$ be the moduli stack  whose $R$-points are $G$-torsors $(\cE_{0},\cdots, \cE_{s})$ over the disk $D_{R}$ with isomorphisms $\a_{i}: \cE_{i-1}|_{D_{R}^{\times}}\isom\cE_{i}|_{D^{\times}_{R}}$ for all $1\le i\le s$. Then we have an evaluation map $\ev^{[s]}: \Hk^{[s]}_{G,\bK}\to \QGr^{[s]}_{G}/\Aut(D)$ and a diagram in which all three squares are Cartesian
\begin{equation*}
\xymatrix{  & \QGr^{[s]}/\Aut(D) \ar[rr]^{\d}\ar[dd]_{\g}& & (\QGr/\Aut(D))^{s}\\
\Hk^{[s]}_{G,\bK}\ar[rr]^{\wt\y}\ar[dd]\ar[ur]^{\ev^{[s]}} && \Hk^{(\{1\},\cdots,\{s\})}_{G,\bK}\ar[ur]^{\ev^{(\{1\},\cdots, \{s\})}}\ar[dd]\\
& \QGr/\Aut(D)\\
\Hk^{\{1\}}_{G,\bK}\ar[ur]^{\ev^{\{1\}}}\ar[rr]^{\y} && \Hk^{I}_{G,\bK}
}
\end{equation*}
The map $\d$ is induced from $\QGr^{[s]}\to (\QGr)^{s}$, which sends $((\cE_{i})_{0\le i\le r},(\a_{i})_{1\le i\le r})$ to $(\cE_{i-1},\cE_{i}, \a_{i})_{1\le i\le s}$; the map $\g$ is induced from $\QGr^{[s]}\to \QGr$ that  sends $((\cE_{i})_{0\le i\le r},(\a_{i})_{1\le i\le r})$ to $(\cE_{0},\cE_{s}, \a_{r}\c\cdots\c\a_{1})$. 

Again by proper base change, \eqref{IC fact b} reduces to the isomorphism
\begin{equation*}
\g_{!}\d^{*}\Sat_{\{1,\cdots , s\}}(V)\cong \Sat(\Res^{\dG^{s}}_{\dG} V).
\end{equation*}
By linearity, it suffices to prove this for $V=\boxtimes_{i=1}^{s} V_{i}$, in which case the left side is the convolution $\Sat(V_{1})\star\cdots\star\Sat(V_{s})$ in the tensor category $\Perv_{\Aut(D)}(\QGr)$, while the right side is $\Sat(V_{1}\ot\cdots\ot V_{s})$. The desired isomorphism comes from the monoidal structure of the Satake equivalence.
\end{proof}

\subsection{Partial Frobenius}\label{sss:pF}
Let $I=I_{1}\coprod\cdots\coprod I_{r}$ be a decomposition into {\em singletons}. This gives a bijection $\s: I\isom \{1,\cdots, r\}$ such that  $I_{j}=\{\s^{-1}(j)\}$.  Recall the partial Frobenius map $\Fr^{\pl}$ defined in \eqref{pl Fr}. 
We have a commutative diagram
\begin{equation}\label{pl Frob ev}
\xymatrix{ \Sht^{(I_{1},\cdots, I_{r})}_{G,\bK}\ar[d]^{\Fr^{\pl}}\ar[rr]^-{\ev^{(I_{1},\cdots, I_{r})}} && (\QGr/\Aut(D))^{I}\ar[d]^{(\Fr, \id,\cdots, \id)}\\
\Sht^{(I_{2},\cdots, I_{r}, I_{1})}_{G,\bK}\ar[rr]^-{\ev^{(I_{2},\cdots, I_{1})}} && (\QGr/\Aut(D))^{I}}
\end{equation}
Here the right vertical map is the Frobenius on the first factor $I_{1}$, and the identity map for the other factors.

For the irreducible representation $V_{\l}\in \Rep(\dG)$ with highest weight $\l$, fix a Weil structure on $\Sat(V_{\l})$ whose restriction on $\Gr_{G,\l}$ is $\Qlbar[\j{2\r, \l}](\j{\r,\l})$ (which involves a choice of $q^{1/2}$ if $\j{\r,\l}$ is a half integer). Extend these Weil structures by additivity to $\Sat(V)$ for all $V\in \Rep(\dG)$. More generally these Weil structures extend to Weil structures on $\Sat_{I}(V)\in \Perv((\QGr/\Aut(D))^{I})$ for any finite set $I$ and $V\in \Rep(\dG^{I})$, compatibly with external tensor products.

Now let $V=\boxtimes_{i\in I}V_{i}\in \Rep(\dG^{I})$. We denote $V_{\s^{-1}(j)}$ also by $V_{j}$, for $1\le j\le r$. From the diagram \eqref{pl Frob ev} and using the Weil structure on $\Sat(V_{1})$ fixed above, we get an isomorphism
\begin{equation}\label{pl Fr pullback IC}
\IC^{(I_{1},\cdots, I_{r})}_{\Sht}(V)\cong \Fr^{\pl*}\IC^{(I_{2},\cdots, I_{1})}_{\Sht}(V).
\end{equation}
On the other hand, proper base change for the topologically Cartesian diagram \eqref{pl Fr UI} gives an isomorphism
\begin{equation*}
\Fr_{I_{1}}^{*}\cH^{I}(V)= \Fr_{I_{1}}^{*}\pi^{(I_{2},\cdots, I_{1})}_{!}\IC^{(I_{2},\cdots, I_{1})}_{\Sht}(V)\isom  \pi^{(I_{1},\cdots, I_{r})}_{!}\Fr^{\pl*}\IC^{(I_{2},\cdots, I_{1})}_{\Sht}(V)
\end{equation*}
Combined with \eqref{pl Fr pullback IC} we get an isomorphism
\begin{equation*}
\phi_{I_{1}}: \Fr_{I_{1}}^{*}\cH^{I}(V)\cong \pi^{(I_{1},\cdots, I_{r})}_{!}\Fr^{\pl*}\IC^{(I_{2},\cdots, I_{1})}_{\Sht}(V)\cong \pi^{(I_{1},\cdots, I_{r})}_{!}\IC^{(I_{1},\cdots, I_{r})}_{\Sht}(V)=\cH^{I}(V).
\end{equation*}
This isomorphism is independent of the decomposition of $I-I_{1}$ into singletons.  Also, such isomorphisms extend by additivity to all $V\in \Rep(\dG^{I})$.

Since we can choose the bijection $\s$ so that $I_{1}$ is any given element $i\in I$, we a canonical isomorphism for all $i\in I$ and $V\in\Rep(\dG^{I})$:
\begin{equation*}
\phi_{i}:  \Fr_{i}^{*}\cH^{I}(V)\isom  \cH^{I}(V).
\end{equation*}
Here $\Fr_{i}: U^{I}\to U^{I}$ is the Frobenius on the $i$-th factor and the identity elsewhere.

The following is easy to check from the definition.
\begin{prop}
The isomorphisms $(\phi_{i})_{i\in I}$ commute with each other up to canonical isomorphisms, and their composition is the canonical Weil structure on $\cH^{I}(V)$, i.e., an isomorphism $\Fr^{*}_{U^{I}}\cH^{I}(V)\isom  \cH^{I}(V)$ induced from the Weil structure on $\Sat_{I}(V)$.
\end{prop}

In particular, for any subset $J\subset I$, we have a canonical isomorphism $\phi_{J}: \Fr_{J}^{*}\cH^{I}(V)\isom  \cH^{I}(V)$, where $\Fr_{J}: U^{I}\to U^{I}$ is the Frobenius on $U^{J}$ and the identity on the other factors. It is defined as the composition of $\phi_{i}, i\in J$ in any order.

\section{Advanced topics}\label{s:adv}
In this section we give a survey of several aspects of the moduli of Shtukas that deserve further study. 

\subsection{Compactification}
The non-properness of the map $\pi^{(I_{1},\cdots, I_{r}),\le \un\l}_{G,\bK}: \Sht^{(I_{1},\cdots, I_{r}),\le \un\l}_{G,\bK}\to U^{I}$ causes difficulty in studying its direct image complex, and in particular the cohomology of its generic fiber. For example, had it been proper, the complex $\cH^{I}(V)$ would be automatically pure, which would have consequences on the Ramanujan-Petersson conjecture for automorphic forms over function fields. 

Since $\pi^{(I_{1},\cdots, I_{r}),\le \un\l}_{G,\bK}$ is not of finite type, to compactify it, one has to restrict to a Harder-Narasimhan truncation ${}^{\le \mu}\Sht^{(I_{1},\cdots, I_{r}),\le \un\l}_{G,\bK}$ as in \S\ref{sss:HN}.

In \cite{Dr-comp}, for the Drinfeld example \ref{ex:Dr} in  the case $G=\GL_{2}$, Drinfeld constructed a compactification of (the truncations of) $\Sht^{(I_{1},I_{2}),(\l_{1},\l_{2})}_{\GL_{2}, \bK}$. In \cite{Laff-comp}, when there is no level structure,  the construction of Drinfeld was generalized to $\GL_{n}$. In \cite{LLaff}, a compactification for the Drinfeld example \ref{ex:Dr} with level structures was constructed (with less preferable properties) and used to prove the global Langlands correspondence for function fields. In \cite{NDT} Ng\^o Dac Tuan gave an alternative construction of the compactification using geometric invariant theory. 

The basic idea of these constructions is to relax the isomorphism $\io: \cE_{r}\isom {}^{\t}\cE_{0}$ in the modular description in \S\ref{sss:Sht points}, and replace it with a point in the Vinberg monoid attached to $G$. In the case $G=\GL_{n}$, such a point is commonly known as a complete homomorphism between two vector spaces or vector bundles. For general $G$, an unpublished manuscript of Ng\^o Dac Tuan and Varshavsky gives a construction of the compactifications of the moduli of Shtukas along these lines. 

To the best of the author's knowledge, it is an open question to describe the singularities of the compactification of $\Sht^{(I_{1},\cdots, I_{r}),\le\un\l}_{G,\bK}$.


\subsection{Algebraic cycles}
We give a brief survey of examples of algebraic cycles on the moduli stack of Shtukas. For simplicity we do not consider level structures in the discussions below.

\sss{Horocycles} The definition of the iterated Hecke stack given in \S\ref{sss:it Hk} makes sense for any algebraic group, not necessarily reductive. Therefore the moduli stack of iterated Shtukas is also defined for an arbitrary algebraic group over $k$. We apply this construction to a parabolic subgroup $P\subset G$ with Levi quotient $L$. For a finite set $I$ with a decomposition $I=I_{1}\coprod \cdots\coprod I_{r}$ and a test scheme $S$ over $k$, $\Sht^{(I_{1},\cdots, I_{r})}_{P}(S)$ classifies tuples
$$((x_{i})_{i\in I}, (\cE_{P,j})_{0\le j\le r}, (\a_{j}: \cE_{P, j-1}\dr\cE_{P, j})_{1\le j\le r}, \io)$$
Here the only difference with the description of $\Sht^{(I_{1},\cdots, I_{r})}_{G}(S)$ in \S\ref{sss:Sht points} is that $\cE_{P,j}$ are $P$-torsors over $X$. By the functoriality of the construction with respect to the change of groups, the canonical maps $G\leftarrow{P}\to L$ induce maps
\begin{equation*}
\xymatrix{ \Sht_{G}^{(I_{1},\cdots, I_{r})} & \ar[l]_{\frp}\Sht_{P}^{(I_{1},\cdots, I_{r})} \ar[r]^{\frq}& \Sht_{L}^{(I_{1},\cdots, I_{r})}.
}
\end{equation*}
The map $\frp$ turns $\cE_{P,j}$ into the induced $G$-torsor $\cE_{j}=\cE_{P,j}\times^{P}G$, and $\frq$ turns $\cE_{P,j}$ into the induced $L$-torsor $\cE_{L,j}=\cE_{P,j}/U_{P}$ (where $U_{P}$ is the unipotent radical of $P$).

To define a bounded version of $\Sht^{(I_{1},\cdots, I_{r})}_{P}$, we simply take an admissible tuple $\un\mu=(\mu_{i})_{i\in I}$ of dominant coweights for $L$, and let 
\begin{equation*}
\Sht_{P}^{(I_{1},\cdots, I_{r}), \le\un\mu}=\frq^{-1}(\Sht_{L}^{(I_{1},\cdots, I_{r}), \le\un\mu} ). 
\end{equation*}
It was shown in \cite[Proposition 5.7]{Var} that, if we restrict over the {\em generic point $\y^{I}$} of $U^{I}$,  the map 
$$\frp^{\le\un\mu}_{\y^{I}}: \Sht_{P,\y^{I}}^{(I_{1},\cdots, I_{r}), \le\un\mu}\to \Sht_{G,\y^{I}}^{(I_{1},\cdots, I_{r})}$$ 
obtained as the restriction of $\frp$ is finite and unramified. The image of the fundamental cycle of $\Sht_{P,\y^{I}}^{(I_{1},\cdots, I_{r}), \le\un\mu}$ define {\em horocycles} on $\Sht_{G,\y^{I}}^{(I_{1},\cdots, I_{r})}$.

Roughly speaking, the cycle classes of horocycles are responsible for the Eisenstein part of the cohomology of $\Sht_{G}^{(I_{1},\cdots, I_{r})}$. For the $\GL_{2}$ example in \ref{ex:Dr}, the geometry of the horocycles were analyzed in \cite[\S4]{Dr-FBun} and its relation with Eisenstein series was studied in detail \cite[\S3]{Dr-Pet}.

It is an open question in general to describe the Eisenstein part of the cohomology of $\Sht_{G}^{(I_{1},\cdots, I_{r}),\le\un\l}$ in the style of the spectral decomposition of the space of automorphic forms. See \cite[\S7]{YZ1} for a qualitative statement in the case $G=\GL_{2}$, and a finiteness result for general $G$ in \cite{X}.

\sss{Special cycles}

Special cycles on Shimura varieties are given by Shimura subvarieties, i.e., they come from reductive subgroups of $G$. We can similarly consider special cycles on the moduli stack of Shtukas. Namely if $\ph: H\to G$ is a homomorphism of reductive groups over $k$, by the functoriality of the notion of Shtukas with respect to change of groups, it induces a map of moduli stacks
$$\th: \Sht^{(I_{1},\cdots, I_{r})}_{H}\to \Sht^{(I_{1},\cdots, I_{r})}_{G}.$$ 
For the bounded version one can start with a tuple $\un\mu$ of dominant coweights for $H$ and choose $\un\l$ for $G$ large enough so that $\th$ sends $\Sht^{(I_{1},\cdots, I_{r}), \le\un\nu}_{H}$ to $\Sht^{(I_{1},\cdots, I_{r}), \le\un\l}_{G}$.  It is proved by Breutmann \cite{Br} that $\th$ is schematic, finite and unramified, in the generality where one is allowed to impose parahoric level structures on $\Sht^{(I_{1},\cdots, I_{r})}_{H}$.  See also \cite{Y-cycle} for an alternative proof. The image of $\th$ therefore defines an element in the Chow group of $\Sht^{(I_{1},\cdots, I_{r}),\le\un\l}_{G}$ for suitable $\un\l$.

We give some examples of special cycles of arithmetic significance.  Some of these examples don't strictly fit into the situation above.

\begin{exam}[Heegner-Drinfeld cycles]\label{ex:HD} Let $\nu: X'\to X$ be an unramified double cover. Let $r$ be even and let $I=\{1,2,\cdots, r\}$ be decomposed into singletons $I_{i}=\{i\}$. Let $\un\mu=(\mu_{i})_{1\le i\le r}$, where each $\mu_{i}$ is either $(1,0,\cdots,0)$ or $(0,\cdots, 0,-1)$ (minuscule coweights of $\GL_{n}$). Let $\Sht^{\un\mu}_{n,X'}$ be the moduli stack $\Sht^{(I_{1},\cdots, I_{r}),\le\un\mu}_{\GL_{n},X'}$ of iterated rank $n$ Shtukas over $X'$ with bounds given by $\un\mu$. 

Let $\un\l=(\l_{i})_{1\le i\le r}$ be the collection of minuscule dominant coweights of $\GL_{2n}$ determined by $\un\mu$: if $\mu_{i}=(\underbrace{0,\cdots, 0}_{n-1},-1)$ then $\l_{i}=(\underbrace{0,\cdots,0}_{2n-1}, -1)$; if $\mu_{i}=(1,\underbrace{0,\cdots,0}_{n-1})$ then $\l_{i}=(1, \underbrace{0,\cdots,0}_{2n-1})$. Similarly denote by $\Sht^{\un\l}_{2n,X}$ the moduli stack $\Sht^{(I_{1},\cdots, I_{r}),\le\un\l}_{\GL_{2n},X}$ of iterated rank $2n$ Shtukas over $X$ with bounds given by $\un\l$. 

We have a map
\begin{equation*}
\th^{\un\mu}_{\HD(n)}: \Sht^{\un\mu}_{n, X'}\to \Sht^{\un\l}_{2n,X}
\end{equation*}
sending a rank $n$ Shtuka $((x'_{i}), (\cE_{i}), (\a_{i}), \io)$ over $X'$ to the rank $2n$ Shtuka $((\nu(x'_{i})), (\nu_{*}\cE_{i}),(\nu_{*}\a_{i}), \nu_{*}\io)$ over $X$. We call the image of the fundamental cycle under $\th^{\un\mu}_{\HD(n)}$ the {\em Heegner-Drinfeld cycle}. Since $\dim \Sht^{\un\mu}_{n,X'}=rn$ while $\dim \Sht^{\un\mu}_{2n,X}=2rn$ by Proposition \ref{p:local model}(2), we have
\begin{equation*}
\dim  \Sht^{\un\mu}_{n,X'}=\frac{1}{2}\dim \Sht^{\un\l}_{2n,X}.
\end{equation*}
Therefore the image of the fundamental class of $\Sht^{\un\mu}_{n,X'}$ under $\th^{\un\mu}_{\HD(n)}$ define a middle dimensional cycle on $\Sht^{\un\l}_{2n,X}$.

To see how this cycle is related to Heegner points, we consider the case $n=1$ and a variant of the above construction with one leg. Let $\infty\in |X|$ be a closed point that is inert in $X'$, so that $\infty'\in |X'|$ is the unique closed point over $\infty$.  Consider the Drinfeld modular (relative) curve $\Sht^{\Dr}_{2}$ in \S\ref{sss:one leg}. Let $\un\mu=(\mu_{1},\mu_{2})$ where $\mu_{1}=-1$ and $\mu_{2}=1$ as coweights of $\Gm$. Let $\cP_{X'}=\Sht^{(\{1\},\{2\}), \un\mu}_{\GL_{1}, X'}|_{U'\times \{\infty'\}}$; this is a $\Pic_{X'}(k)$-torsor over $U'=X'-\infty'$. We construct a map
\begin{equation*}
\th^{1}_{\HD(1)}: \cP_{X'}\to \Sht^{\Dr}_{2}.
\end{equation*}
For an $S$-point $\cL_{0}\leftarrow\cL_{0}(-x')\rightarrow\cL_{0}(-x'+y')\isom {}^{\t}\cL_{0}$ (where $y':S\to \infty'$) of $\cP_{X'}$, consider the rank two bundle $\cE=\nu_{*}\cL_{0}$ on $X'$ whose fiber over $S\times\infty$ has a splitting $\cE|_{S\times\infty}=\cE_{y'}\op \cE_{\s(y')}$ (where $\s\in \Gal(X'/X)$ is the nontrivial element). We equip $\cE|_{S\times \infty}$ with the full flag $\cE_{y'}\subset \cE_{S\times\infty}$, and define $\cE(-\frac{1}{2})\subset \cE$ to be the preimage of $\cE_{y'}$ under the restriction map $\cE\to \cE|_{S\times \infty}$. Then $\cE(-\frac{1}{2})=\nu_{*}\cL_{0}(-\s(y'))$. Hence ${}^{\t}(\cE(-\frac{1}{2}))=\nu_{*}({}^{\t}\cL_{0})(-y')\cong \nu_{*}\cL_{0}(-x')$ (using that $y'\c\Fr_{S}=\s(y')$ because $\infty$ is inert). Let $\phi: {}^{\t}(\cE(-\frac{1}{2}))\to \cE$ be induced by the embedding $\cL_{0}(-x')\incl \cL_{0}$. This defines an $S$-point $(\xi=\nu(x'), \cE, \{\cE(\frac{i}{2})\}, \phi)\in \Sht^{\Dr}_{2}(S)$. 

Using the dictionary between elliptic modules and Shtukas in \S\ref{sss:dict},  $\cP_{X'}$ classifies rank two elliptic modules on $X$ with endomorphisms by $A'=\G(U', \cO_{U'})$, which is a function field analog of elliptic curves with complex multiplication. Hence $\th(\cP_{X'})$ is the function field analog of Heegner points on modular curves. Generalizations of this construction with Iwahori level structures is carried out in \cite{YZ2}.
\end{exam}

\begin{exam}[Gan-Gross-Prasad cycles]\label{ex:GGP} 
Let $\nu: X'\to X$ be an unramified double cover. Let $\Ug_{n}$ be a unitary group scheme over $X$ with respect to $X'$. Without giving details about the nonsplit group $\Ug_{n}$ itself, let us spell out the moduli meanings of $\Bun_{\Ug_{n}}$ and $\Sht^{r}_{\Ug_{n}}$.

For a test scheme $S$ over $k$,  a rank $n$ \emph{Hermitian} (also called \emph{unitary}) bundle on $X \times S$ with respect to $\nu:X'\to X$ is a vector bundle $\cF$ of rank $n$ on $X' \times S$, equipped with an isomorphism $h : \cF \xrightarrow{\sim} \sigma^* \cF^{\vee}$ such that $\sigma^* h^{\vee} = h$. Here $\cF^{\vee}=\un\Hom(\cF, \om_{X'})$ is the Serre dual of $\cF$. Then $\Bun_{\Ug_{n}}$ is the stack whose $S$-points are the groupoid of rank $n$ Hermitian bundle on $X \times S$ with respect to $\nu$. 

An elementary lower modification of a rank $n$ Hermitian bundle $(\cF,h)$ on $X'\times S$ along $x': S\to X'$ is another rank $n$ Hermitian bundle $(\cF',h')$ on $X'\times S$ together with an isomorphism $\a:  \cF|_{X' \times S - \Gamma(x') - \Gamma(\sigma(x'))} \xrightarrow{\sim} \cF'|_{X' \times S - \Gamma(x')-\Gamma(\sigma(x'))}$, compatible with the Hermitian structures, with the following property: there exists a rank $n$ vector bundle $\cF^{\flat}$ (necessarily unique) such that $f$ can be factored as
\begin{equation*}
\xymatrix{
\cF &\ar@{_{(}->}[l]_-{\oll{\a}}\cF^{\flat}   \ar@{^{(}->}[r]^-{\orr{\a}} & \cF'}
\end{equation*}
such that $\coker(\oll{\a})$ is locally free of rank $1$ over $\Gamma(x')$, and $\coker(\orr{\a})$ is locally free of rank $1$ over $\Gamma(\sigma(x'))$. 

Let $r$ be an even integer. Define a stack $\Sht_{\Ug_{n}}^{r}$ whose $S$-points are given by the groupoid of the following data
$$((x'_{i})_{1\le i\le r}, (\cF_{i}, h_{i})_{0\le i\le r}, (\a_{i})_{1\le i\le r}, \io)$$
where $x'_{i}: S\to X'$, $(\cF_{i},h_{i})$ are rank $n$ Hermitian bundles over $X'\times S$, $\a_{i}: \cF_{i-1}\dr\cF_{i}$ is an elementary lower modification of $\cF_{i-1}$ along $x'_{i}$, and $\io$ is an isomorphism of Hermitian bundles $\cF_{r}\cong {}^{\t}\cF_{0}$. 

Now fix a rank one Hermitian bundle $\cL$ on $X'$ with respect to $\nu$, which is the same thing as a line bundle together with an isomorphism $\Nm_{X'/X}\cL\cong \om_{X}$.   We have the map
\begin{equation*}
\ka_{\cL}:\Sht^{r}_{\Ug_{n}}\to \Sht^{r}_{\Ug_{n+1}} 
\end{equation*}
sending $((x'_{i}), (\cF_{i}), \cdots)$ to $((x'_{i}), (\cF_{i}\op\cL),\cdots)$. Consider the diagonal map
\begin{equation*}
\th^{r}_{\GGP(n)}=(\id,\ka_{\cL}): \Sht^{r}_{\Ug_{n}}\to \Sht^{r}_{\Ug_{n}}\times_{X'^{r}}\Sht^{r}_{\Ug_{n+1}}.
\end{equation*}
By \cite[Lemma 6.9(2)]{FYZ1}, the map $\Sht^{r}_{\Ug_{n}}\to X'^{r}$ is smooth of relative dimension $r(n-1)$. Therefore 
$$\dim\Sht^{r}_{\Ug_{n}}=\frac{1}{2}(\dim\Sht^{r}_{\Ug_{n}}\times_{X'^{r}}\Sht^{r}_{\Ug_{n+1}}).$$
The image of the fundamental cycle of $\Sht^{r}_{\Ug_{n}}$ under $\th^{r}_{\GGP(n)}$ defines a middle dimensional cycle on the target. We denote its cycle class by $Z^{r}_{\GGP(n)}\in \cohog{2rn}{(\Sht^{r}_{\Ug_{n}}\times_{X'^{r}}\Sht^{r}_{\Ug_{n+1}})_{\ov k}}(rn)$.

Inspired by the arithmetic Gan-Gross-Prasad conjecture \cite{GGP}, we expect that for a cuspidal automorphic representation $\Pi$ of $\Ug_{n}(\AA_{F})\times \Ug_{n+1}(\AA_{F})$, the self-intersection number of the $\Pi$-isotypic component of $Z^{r}_{\GGP(n)}$ to be well-defined (note a priori it is not a cycle with proper support), and to be related to the $r$th central derivative of the Rankin-Selberg  $L$-function of the base change of $\Pi$ to $\GL_{n}(\AA_{F'})\times\GL_{n+1}(\AA_{F'})$, where $F'=k(X')$. 

\end{exam}

\begin{exam}[Kudla-Rapoport cycles]\label{ex:KR}
In \cite{FYZ1} we introduced a class of algebraic cycles that are function field analogs of the Kudla-Rapoport special cycles \cite{KR2} on unitary Shimura varieties.

We are in the setting of Example \ref{ex:GGP}. Let $\cE$ be a rank $m$ vector bundle on $X'$. We define the stack $\cZ_{\KR(n)}^{r}(\cE)$ whose $S$-points are the groupoid of the following data: 
 \begin{itemize}
 \item A Hermitian Shtuka $((x'_i)_{1\le i\le r}, (\cF_{i},h_{i})_{0\le i\le r}, (\a_{i})_{1\le i\le r}, \io) \in  \Sht_{\Ug_{n}}^{r}(S)$.
 \item Maps of coherent sheaves  $t_i : \cE \boxtimes \cO_{S} \to \cF_i$ on $X'\times S$ for $i=0,1,\cdots, r$, such that the isomorphism $\io : \cF_r \cong {}^{\t}\cF_0$ intertwines $t_r$ with ${}^{\t}t_0$, and the maps $t_{i-1}, t_{i}$ are intertwined by the modification $\a_i :  \cF_{i-1} \dr \cF_{i}$ for each $i = 1, \cdots, r$, i.e., the diagram below commutes. 
\begin{equation}\label{EFsq}
\xymatrix{\cE\boxtimes\cO_{S} \ar[d]^{t_{0}} \ar@{=}[r] &  \cE\boxtimes\cO_{S} \ar@{=}[r]  \ar[d]^{t_1} & \cdots \ar[d] \ar@{=}[r] & \cE\boxtimes\cO_{S} \ar[r]^{\sim} \ar[d]^{t_{r}} & {}^{\t}(\cE \boxtimes\cO_{S})\ar[d]^{{}^{\t}t_{0}}\\
\cF_0 \ar@{-->}[r]^{\a_0} & \cF_1 \ar@{-->}[r]^{\a_1}  &  \cdots \ar@{-->}[r]^{\a_r} &  \cF_{r} \ar[r]^-{\io}_{\sim} & {}^{\t}\cF_0 
}
\end{equation}
\end{itemize}


For $((x'_i), (\cF_i,h_{i}), (\a_i), \io, (t_i)) \in \cZ_{\KR(n)}^r(\cE)(S)$, consider the composition
\begin{equation*}
\cE \boxtimes \cO_S \xr{t_i} \cF_i \xr{h_i} \s^* \cF_i^{\vee} \xr{\s^* t^{\vee}_{i}} \s^* \cE^{\vee} \boxtimes \cO_S.
\end{equation*}
By the commutative squares in \eqref{EFsq}, the above map is independent of $i$, and it descends to a map $a: \cE\to \s^{*}\cE^{\vee}$ satisfying $\s^{*}a^{\vee}=a$. In other words, every point of $\cZ^{r}_{\KR(n)}(\cE)$ induces a possibly degenerate Hermitian form on $\cE$, according to which we have a decomposition into open and closed substacks
\begin{equation*}
\cZ_{\KR(n)}^{r}(\cE)=\coprod_{a\in \Herm(\cE)} \cZ^{r}_{\KR(n)}(\cE,a).
\end{equation*}
where $\Herm(\cE)$ is the set of (possibly degenerate) Hermitian forms on $\cE$. We have a natural map
\begin{equation*}
\th_{\KR(n)}^{r}(\cE,a): \cZ^{r}_{\KR(n)}(\cE,a)\to \Sht^{r}_{\Ug_{n}}
\end{equation*}
that is finite and unramified. 

The expected dimension of $\cZ^{r}_{\KR(n)}(\cE,a)$ is $r(n-m)$. However, its actual dimension is often bigger. In \cite{FYZ2} we use derived algebraic geometry to equip $\cZ^{r}_{\KR(n)}(\cE,a)$ with a virtual fundamental class $[\cZ^{r}_{\KR(n)}(\cE,a)]$ of the expected dimension $r(n-m)$; its image under $\th^{r}_{\KR(n)}(\cE,a)$ is a function field analog of Kudla-Rapoport special cycles on unitary Shimura varieties.
\end{exam}

\subsection{Intersection theory and higher derivatives of $L$-functions}
Classically, period integrals of automorphic forms can often be expressed as special values of automorphic $L$-functions. In the geometric context, special cycles on Shimura varieties or the moduli stack of Shtukas can be thought of analogs of period integrals. Over number fields, heights of special cycles on Shimura varieties are sometimes related to derivatives of automorphic $L$-functions, such as in the Gross-Zagier formula or in the arithmetic Gan-Gross-Prasad conjecture. Over function fields, there is not only an analogous story but also a generalization: the flexibility of having multiple legs for Shtukas allows one to relate intersection numbers of special cycles to higher derivatives of $L$-functions. We give a brief survey of results in this direction. We refer to \cite{Y-ICM} for more detailed discussion of the motivation and techniques used in the proofs.

\sss{Higher Gross-Zagier formula}

The famous Gross-Zagier formula gives an equality between heights of Heegner points (projected to a Hecke eigen piece) on a modular curve in terms of the first derivative of the automorphic $L$-function of a modular form. This formula has numerous applications in number theory, especially in the arithmetic of elliptic curves. 

In \cite{YZ1} and \cite{YZ2}, we give a higher derivative/multiple legs  version of the function field analog of the Gross-Zagier formula. We first state the version without level structure, which was treated in \cite{YZ1}.

Let $G=\PGL_{2}$, $I=\{1,\cdots, r\}$ decomposed into singletons $I_{i}=\{i\}$, and let $\l_{i}$ be the unique dominant minuscule coweight of $G$. We denote the resulting moduli stack of Shtukas $\Sht^{(I_{1},\cdots, I_{r}),\le\un\l}_{G}$ simply by $\Sht^{r}_{G}$. The admissibility condition on $\un\l$ forces $r$ to be even.  

Let $\nu:X'\to X$ be an unramified double cover. Let $T=(\Res_{X'/X}(\Gm\times X'))/\Gm$ be the one-dimensional nonsplit torus on $X$ that splits over $X'$. For such a group scheme $T$ over $X$, and a sequence $\un\mu=(\mu_{1},\cdots, \mu_{r})\in\{\pm1\}^{r}$,  $\Sht_{T}^{\un\mu}:=\Sht_{T}^{(I_{1},\cdots, I_{r}), \un\mu}$ is defined. Concretely, it can be identified with the quotient $\Sht^{(I_{1},\cdots, I_{r}), \un\mu}_{\GL_{1}, X'}$ (where the base curve is $X'$, see Example \ref{ex:Gm}) by the action of $\Pic_{X}(k)$ by pulling back to $X'$ and tensoring with rank one Shtukas. As in Example \ref{ex:HD}, we have a natural map
\begin{equation*}
\th: \Sht^{\un\mu}_{T}\to \Sht^{r}_{G}.
\end{equation*}
Consider the lifting of $\th$
\begin{equation*}
\th': \Sht^{\un\mu}_{T} \to \Sht'^{r}_{G}:=\Sht^{r}_{G}\times_{X^{r}}X'^{r}.
\end{equation*} 
One easily sees that $\Sht^{\un\mu}_{T}$ is proper of dimension $r$, and $\Sht'^{r}_{G}$ is smooth of dimension $2r$ by Corollary \ref{c:Sht sm}. Therefore the Heegner-Drinfeld cycle $\cZ^{\un\mu}_{T}:=\th'_{*}[\Sht^{\un\mu}_{T}]$ is a middle-dimensional proper cycle in the smooth Deligne-Mumford stack $\Sht'^{r}_{G}$. Let $Z^{\un\mu}_{T}\in \cohoc{2r}{\Sht'^{r}_{G,\kbar},\Qlbar}(r)$ be the cycle class of $\cZ^{\un\mu}_{T}$.

Now let $\pi$ be an everywhere unramified cuspidal automorphic representation of $G(\AA)$ with coefficients in $\Qlbar$.  One can make sense of the projection of $Z^{\un\mu}_{T}$ to the $\pi$-isotypic summand under the Hecke algebra action. We denote the resulting cohomology class by $Z^{\un\mu}_{T,\pi}\in \cohoc{2r}{\Sht'^{r}_{G,\kbar},\Qlbar}(r)$ .

\begin{theorem}[\cite{YZ1}]\label{th:YZ1} The self-intersection number of the cycle class $Z^{\un\mu}_{T,\pi}$ is
\begin{equation*}
\j{Z^{\un\mu}_{T,\pi}, Z^{\un\mu}_{T,\pi}}_{\Sht'^{r}_{G}}=\frac{q^{2-2g}}{2(\log q)^{r}}\frac{\sL^{(r)}(\pi_{F'}, 1/2)}{L(\pi,\Ad,1)}
\end{equation*}
where
\begin{itemize}
\item $\pi_{F'}$ is the base change of $\pi$ to $F'=k(X')$. 
\item $\sL(\pi_{F'},s)=q^{4(g-1)(s-1/2)}L(\pi_{F'},s)$ is the normalized $L$-function of $\pi_{F'}$ such that $\sL(\pi_{F'},s)=\sL(\pi_{F'},1-s)$.
\end{itemize}
\end{theorem}

In \cite{YZ2},  we extended the above theorem to allow the automorphic representation $\pi$ to have Iwahori level structures (which means the local representations $\pi_{v}$ are either unramified or an unramified twist of the Steinberg representation),  and to allow ramifications for the double cover $\nu:X'\to X$. This generality allows applications to the Birch and Swinnerton-Dyer conjecture over function fields in future work, see \cite[\S1.3]{YZ2} for a preview.


\sss{Higher Siegel-Weil formula and higher theta series}
The classical Siegel-Weil formula relates the special values of Siegel-Eisenstein series on the symplectic group (resp. the unitary group) to theta functions, which are generating series of representation numbers of quadratic (resp. Hermitian) forms over number fields.  In \cite{Kud} Kudla began to study an arithmetic version of the Siegel-Weil formula:  he defined an ``arithmetic theta function''-- a generating series of arithmetic cycles on an integral model of a Shimura curve -- and discovered its relationship with the first central derivative of a Siegel-Eisenstein series on $\Sp_{4}$.  In a series of papers \cite{KR1} and \cite{KR2}, Kudla and Rapoport developed this paradigm by defining a generating series of special cycles (Kudla-Rapoport cycles) on integral models of Shimura varieties for orthogonal and unitary groups. They conjectured a relationship between the arithmetic intersection numbers of these cycles to the non-singular Fourier coefficients of the central derivative of the Siegel-Eisenstein series. Their conjecture has been recently proved by Chao Li and Wei Zhang in \cite{LZ}.

In \cite{FYZ1}, we prove a higher derivative version of the Siegel-Weil formula for unitary groups over function fields. Consider the special cycles $\cZ^{r}_{\KR(n)}(\cE,a)$ constructed in Example \ref{ex:KR}. We specialize to the case where $\rk (\cE)=n$ and $a: \cE\to\s^{*}\cE^{\vee}$ is generically an isomorphism (we call such an $a$ non-singular). In this case, $\cZ^{r}_{\KR(n)}(\cE,a)$ is a proper scheme over $k$, yet it may not have the expected dimension which is $0$ in this case. We first define a virtual fundamental cycle $[\cZ^{r}_{\KR(n)}(\cE,a)]$, which is a $0$-cycle on the proper scheme $\cZ^{r}_{\KR(n)}(\cE,a)$, so that its degree is well-defined. We prove the following higher Siegel-Weil formula

\begin{theorem} Let $r\in \NN$ be even, $\cE$ be a rank $n$ vector bundle on $X'$ and $a: \cE\to \s^{*}\cE^{\vee}$ be a non-singular Hermitian form on $\cE$. Then we have
\begin{equation*}
\deg [\cZ_{\KR(n)}^r(\cE, a)]=\frac{1}{(\log q)^r}  \left(\dfrac{d}{ds} \right)^r\Big|_{s=0} \left( q^{ds} \wt{E}_{a}(\cE,s,\Phi) \right).
\end{equation*}
Here $d=-\deg(\cE)+n\deg\omega_X =-\chi(X',\cE)$, and $\wt{E}_{a}(\cE,s,\Phi)$ is a certain normalization of the $a$th Fourier coefficient of the Siegel-Eisenstein series $E(g,s,\Phi)$ on the unitary group $\Ug_{2n}$, expanded at $\cE$ (lifted to an element in the Siegel-Levi $\GL_{n}(\AA_{F})$).
\end{theorem}

The proof is completely different from the number field case which was based on induction on $n$. The proof reduces to an equality of perverse sheaves that originated from Springer theory.


In \cite{FYZ2}, we construct the complete generating series of special cycles, i.e., we define a virtual fundamental class $[\cZ_{\KR(n)}^r(\cE, a)]$ for $\cE$ of any rank $m\le n$ and arbitrary Hermitian forms $a$ on $\cE$. In particular, $a$ can be singular or even zero, in which case the virtual  fundamental classes involve Chern classes of tautological bundles on $\Sht^{r}_{\Ug_{n}}$.   Let $Z_{\KR(n)}^{r}(\cE,a)$ be the image of $[\cZ_{\KR(n)}^r(\cE, a)]$ under the map $\th^{r}_{\KR(n)}(\cE,a)$, as an element in the Chow group $\Ch_{r(n-m)}(\Sht^{r}_{\Ug_{n}})$. We then assemble these algebraic cycles into a higher theta series. 

More precisely, for fixed $m\le n$, then we consider the quasi-split unitary group $\Ug_{2m}$ over $X$ with respect to the double cover $X'/X$, and its standard Siegel parabolic $P_m$. We define a function
\begin{equation*}
\wt Z^{r}_{m}: \Bun_{P_{m}}(k)\to \Ch_{r(n-m)}(\Sht^{r}_{\Ug_{n}})
\end{equation*}
as follows. A point in $\Bun_{P_{m}}(k)$ is the same datum as a pair  $(\cG,\cE)$, where $\cG$ is a rank $2m$ Hermitian vector bundle on $X'$ (here the notion of a Hermitian form is an isomorphism $\cG\isom \s^{*}\cG^{*}$, where $\cG^{*}$ is the {\em linear dual}  of $\cG$ rather than Serre dual), and $\cE$ is a Lagrangian sub-bundle of $\cG$. Then we define $\wt Z^{r}_{m}(\cG,\cE) $ using the ``Fourier expansion" 
\begin{equation*}
\wt Z^{r}_{m}(\cG,\cE) =\chi(\det\cE)q^{n(\deg \cE-\deg\om_{X}) /2}\sum_{a\in \Herm(\cE)}\psi_{0}(\j{e_{\cG,\cE}, a})Z^{ r}_{\KR(n)}(\cE, a).
\end{equation*}
Here $\psi_{0}:k\to \Qlbar^{\times}$ is a nontrivial character, $e_{\cG,\cE}\in \Ext^{1}(\s^{*}\cE^{*}, \cE)$ is the extension class of the exact sequence $\cE\incl \cG\surj \s^{*}\cE$, which can be paired with $a\in \Hom(\cE,\s^{*}\cE^{\vee})$ under Serre duality.

\begin{conj}[Modularity conjecture, {\cite[Conjecture 4.12]{FYZ2}}]The function $\wt Z^{r}_{m}$ descends to a function
\begin{equation*}
Z^{r}_{m}:\Bun_{\Ug_{2m}}(k)\to \Ch_{r(n-m)}(\Sht^{r}_{\Ug_{n}}).
\end{equation*}
\end{conj}
In other words, the  class $\wt Z^{r}_{m}(\cG,\cE)\in \Ch_{r(n-m)}(\Sht^{r}_{\Ug_{n}})$ should depend only on the Hermitian bundle $\cG$ and not on its Lagrangian sub-bundle $\cE$. When $r=0$, $\Ch_{0}(\Sht^{0}_{\Ug_{n}})$ is simply the space of functions on $\Bun_{\Ug_{n}}(k)$ and the classical theory of Weil representation implies that in this case $\wt Z^{r}_{m}$ does descend to $\Bun_{\Ug_{2m}}(k)$. The resulting two variable function on $\Bun_{\Ug_{2m}}(k)\times \Bun_{\Ug_{n}}(k)$ is the classical theta function.

In \cite{FYZ2} we give several evidences for the Modularity Conjecture; in \cite{FYZ3} we show the conjecture holds if we restrict to the generic fiber of $\Sht^{r}_{\Ug_{n}}\to X'^{r}$ and pass to cohomology classes of the cycles.

\end{document}